\documentclass[10pt,a4paper]{article}

\usepackage[utf8]{inputenc}
\usepackage[english]{babel}
\usepackage[intlimits]{amsmath}
\usepackage{amsfonts}
\usepackage{amssymb}
\usepackage{lmodern}
\usepackage[thref, amsthm, hyperref, thmmarks, amsmath]{ntheorem}	
\usepackage[T1]{fontenc}
\usepackage{etoolbox}
\usepackage{calc}
\usepackage{mathtools}
\usepackage{geometry}
\usepackage[multiuser]{fixme}

\usepackage[pdftex=true,breaklinks=true,bookmarks=true,linktocpage=true,pdfpagelabels=true,plainpages=false,pagebackref=true, draft=false]{hyperref}

\newcommand{\R}{\mathbb{R}}
\newcommand{\N}{\mathbb{N}}
\newcommand{\Z}{\mathbb{Z}}
\newcommand{\C}{\mathbb{C}}
\newcommand{\LL}{\mathcal{L}}
\newcommand{\GL}{\mathcal{G}}
\newcommand{\EL}{\mathcal{E}}
\newcommand{\XL}{\mathcal{X}}
\renewcommand{\epsilon}{\varepsilon}
\renewcommand{\rho}{\varrho}
\renewcommand{\d}{\ensuremath{\,\mathrm{d}}}

\DeclareMathOperator{\intM}{intM}

\newcommand{\transpose}{\intercal}

\DeclarePairedDelimiter{\abs}{\lvert}{\rvert}
\makeatletter
\newcommand{\norm}{%
    \@ifstar
    \normStar%
    \normNoStar%
}
\newcommand{\set}{%
    \@ifstar
    \setStar%
    \setNoStar%
}
\makeatother

\newcommand{\normNoStar}[2][]{\lVert #2 \rVert_{#1}}
\newcommand{\normStar}[2][]{\left\lVert #2 \right\rVert_{#1}}

\newcommand{\restrict}[2]{\left . #1 \right\rvert_{#2}}
\newcommand{\setStar}[2][]{
    \left\{#2%
    \ifthenelse{\equal{#1}{}}{}{\ \middle\vert \ #1}%
    \right\}%
}
\newcommand{\setNoStar}[2][]{
    \{#2%
    \ifthenelse{\equal{#1}{}}{}{\mathrel{\vert}#1}
    \}
}

\newcommand{\difference}[1]{\Delta_{#1}}
\newcommand{\Qp}{\ensuremath{Q^{(p)}}}
\newcommand{\Qpt}{\ensuremath{\widetilde Q^{(p)}}}

\newcommand{\Rpt}{\ensuremath{\widetilde R^{(p)}}}
\newcommand{\Gp}{\ensuremath{\GL^{(p)}}}

\newcommand{\numberthis}{\refstepcounter{equation}\tag{\theequation}}

\newlength{\adjustedalignmentskip}
\DeclareRobustCommand{\adjustedalignment}[3][1]{
    \setlength{\adjustedalignmentskip}{\widthof{$\displaystyle#2x$}- \widthof{$\displaystyle#3x$}}%
    \setlength{\adjustedalignmentskip}{#1\adjustedalignmentskip}%
    \hspace{\adjustedalignmentskip}#3%
}

\newcommand{\Mpq}[1][p,q]{\ensuremath{\intM^{\left(#1\right)}}}
\newcommand{\Mpz}{\Mpq[p,2]}

\renewcommand{\[}{\begin{equation*}}
\renewcommand{\]}{\end{equation*}}

\theoremstyle{plain}
\newtheorem{lemma}{Lemma}[section]
\newtheorem{theorem}[lemma]{Theorem}
\newtheorem{proposition}[lemma]{Proposition}
\newtheorem{corollary}[lemma]{Corollary}

\renewtheorem*{theorem*}{Theorem}
\renewtheorem*{lemma*}{Lemma}
\renewtheorem*{definition*}{Definition}

\theoremstyle{definition}
\newtheorem{remark}[lemma]{Remark}

\FXRegisterAuthor{d}{dlang}{Daniel}
\FXRegisterAuthor{n}{nlang}{\color{red} Nicole}
\fxusetheme{color}

\fxuselayouts{inline, nomargin}

\newcommand{\mf}{{\frac 5 2}}

\usepackage{authblk}

\author{
    Daniel Steenebr\"ugge\\
        RWTH Aachen University, Institut für Mathematik\\
        Templergraben 55, 52062 Aachen, Germany.\\
        \href{mailto:steenebruegge@instmath.rwth-aachen.de}{steenebruegge@instmath.rwth-aachen.de}\\
        \url{www.instmath.rwth-aachen.de/~steenebruegge/home}
    \and
    Nicole Vorderobermeier\\
        Universität Salzburg, Fachbereich Mathematik\\
        Hellbrunner Strasse 34, 5020 Salzburg, Austria.\\
        \href{mailto:nicole.vorderobermeier@sbg.ac.at}{nicole.vorderobermeier@sbg.ac.at}\\
        \url{uni-salzburg.at/index.php?id=209722}
}

\makeatletter
\newcommand{\subjclass}[2][2020]{%
	\let\@oldtitle\@title%
	\gdef\@title{\@oldtitle\footnotetext{#1 \emph{Mathematics Subject Classification.} #2}}%
}
\newcommand{\keywords}[1]{%
	\let\@@oldtitle\@title%
	\gdef\@title{\@@oldtitle\footnotetext{\emph{Key words and phrases.} #1.}}%
}
\makeatother

\begin{document}
	\title{On the Analyticity of Critical Points of the Generalized Integral Menger Curvature in the Hilbert Case}
		
	\maketitle
	
	    \begin{abstract}
		We prove the analyticity of smooth critical points for generalized integral Menger curvature energies 
		$\textnormal{intM}^{(p,2)}$, with $p \in (\tfrac 73, \tfrac 83)$, subject to a fixed length constraint. 
		This implies, together with already well-known regularity results, that finite-energy, critical $C^1$-curves $\gamma: \R/\Z \to \R^n$ of generalized integral Menger curvature $\textnormal{intM}^{(p,2)}$ subject to a fixed length constraint are not only $C^\infty$ but also analytic. 
		Our approach is inspired by analyticity results on critical points for O'Hara's knot energies based on Cauchy's method of majorants and a decomposition of the first variation.
		The main new idea is an additional iteration in the recursive estimate of the derivatives 
		to obtain a sufficient difference in the order of regularity.
	\end{abstract}

\section{Introduction}

An easy way to produce a knot is to tie a piece of string and glue the ends together. This motivates the mathematical definition of a knot as a continuous embedding of a circle into three-dimensional Euclidean space. Depending on the initial knotting, different knot types arise. We say that two knots are of the same knot class if there exists an ambient isotopy \cite[p.~4f.]{Cromwell:2004} that deforms one knot to the other avoiding any sort of self-intersections or ``cutting and glueing'' phenomena. 
In the following we aim to investigate the regularity of optimal shapes for knots within their knot classes. 

Lead by the question of determining optimal representatives for given polygonal knot classes, Fukuhara \cite{Fukuhara:1988:EnergyKnot} introduced the idea to apply some sort of self-repelling potentials to given knots and follow their gradient flow towards the local minimizer. Based on that concept, O'Hara later established the notion of knot energies by defining them as real-valued functionals on the space of knots that are bounded from below and self-repulsive, i.e.~the energy blows up on a sequence of knots that converges to a curve with self-intersections \cite{OHara:2003:EnergyKnotsConformGeom}.

O'Hara was also the first to suggest such a potential knot energy for actual knots \cite{OHara:1991:EnergyKnot}, motivated from the electrostatic potential energy, in particular the self-avoidance effects described by Coulomb's law. Later, this energy was named Möbius energy due to its invariance with respect to Möbius transformations shown by Freedman, He and Wang \cite{Freedmanetal:1994:MoebiusEnergy}. 

Motivated by applications to molecular biology, Gonzalez and Maddocks elaborated a completely different approach to define knot energies in search of a new notion for thickness of knots \cite{GonzalesMaddocks:1999:GlobalCurvature}.
For any continuous, closed, and rectifiable curve $\gamma:\R/\Z \rightarrow \R^3$, where $\R /\Z \cong S^1$ denotes the circle of length 1, they characterized the thickness of  $\gamma$  by
\begin{align*}
\Delta (\gamma) = \inf_{x,y,z\in \gamma } R(x,y,z) ,
\end{align*}
where $R(x,y,z)$ stands for unique circumcircle radius of three points $x,y,z\in \R^n$ given by 
\begin{align*}
R(x,y,z) = \frac{|y-z|\, |y-x| \, |z-x|}{2 |(y-x)\wedge (z-x)|} = \frac{|y-z|}{2 \sin \measuredangle (y-x,z-x)}.
\end{align*}
Observe that $\Delta(\gamma) = 0$ if $\gamma$ intersects itself. Therefore, the quotient  $\mathcal{L}(\gamma) / \Delta(\gamma)$ of length $\mathcal{L}(\gamma)$ over thickness $\Delta(\gamma)$ called \emph{ropelength} serves as a starting point for new self-repelling potential energies.
We remark that thickness is related to the
local radius of curvature since the circumcircle of three points $x,y,z$ on a sufficiently regular curve $\gamma$ converges to the osculating circle at $x$ for $y$ and $z$ converging to $ x$ on $\gamma$. The radius of the osculating circle at $x$ on $\gamma$  is the inverse of the local curvature. 

The ropelength imposes technical challenges as it consists of operators that are hard to differentiate. These drawbacks can be softened as suggested in \cite{GonzalesMaddocks:1999:GlobalCurvature} by exchanging point-wise maximization with integration, which leads to various \emph{integral Menger\footnote{Named after the Austrian-American mathematician Karl Menger (1902-1985) who used the radius of circumscribed circles on a curve to generalize the concept of curvature to metric spaces \cite{Menger:1930:UntersuchungenUeberAllgemeineMetrik}} curvature energies}
such as
\begin{align}\label{def:intMengcurv}
\mathcal{M}^{p}(\gamma) =   \iiint_{(\R/\Z)^3} \frac{|\gamma'(r)| \, |\gamma'(s)| \, |\gamma'(t)|}{R(\gamma(r),\gamma(s), \gamma(t))^p} \d r \d s\d t.
\end{align}
for any $p>0$. Note that for curves parametrized by arc-length $\lim_{p\rightarrow \infty} \mathcal{M}_p^{1/p} (\gamma) = \frac{1}{\Delta(\gamma)}$. Other variants between ropelength and $\mathcal M^p$ can be found in \cite{GonzalesMaddocks:1999:GlobalCurvature}.

For $p>3$, these functionals are indeed knot energies according to O'Hara's definition, as Strzelecki, Szumańska and von der Mosel showed in \cite{StrzeleckiSzumanskavonderMosel:2013:KnotEnergiesInvolvMengCurv}. Furthermore, they proved a geometric Morrey-Sobolev imbedding in three dimensions, in particular that finite energy of an arc-length parametrized curve implies $C^{1,1-\frac 3p}$-regularity and its image is $C^1$-diffeomorphic to the circle \cite{StrzeleckiSzumanskavonderMosel:2010:RegularizingandselfavoidanceeffectsofintegralMengercurvature}. Blatt completed the characterization of energy spaces in \cite{Blatt:2013:NoteMenger}. Generalizations of the integral Menger curvature to higher-dimensional objects can be found for example in \cite{StrzeleckivonderMosel:2009:IntegralMengercurvatureforsurfaces,BlattKolasinski:2012:SharpboundednessandregularizingeffectsoftheintegralMengercurvatureforsubmanifolds, Kolasinski:2014:GeometricSobolevlikeembeddingusinghighdimensionalMengerlikecurvature}.
For further information on the integral Menger curvature energies  we refer the reader for example to the survey paper \cite{StrzekeckivonderMosel:2013:MengerSurvey}.
 
In order to tackle the main question of the present paper, the regularity of optimal shapes for knots, we first need to choose a suitable knot energy and then conduct regularity studies on its stationary points. 
Now if we chose for this purpose the integral Menger curvature energies $\mathcal{M}^p$, $p>3$, we would arrive at nonlinear, non-local and degenerate Euler-Lagrange operators. A formula for the first variation can be found in \cite{Hermes:2012:Menger} and for its second and third variation in \cite{Knappmann:2020:Menger}. Hence, to produce non-degenerate energies from the integral Menger curvature energies given in \eqref{def:intMengcurv}, Blatt and Reiter \cite{BlattReiter:2015:Menger} suggested to generalize $\mathcal{M}^p$ to

\begin{align}\label{def:GenIntMengCurv}
\textnormal{intM}^{(p,q)}(\gamma) =   \iiint_{(\R/\Z)^3} \frac{|\gamma'(u_1)| \, |\gamma'(u_2)| \, |\gamma'(u_3)|}{R^{(p,q)}(\gamma(u_1),\gamma(u_2), \gamma(u_3))} \d u_1 \d u_2 \d u_3, 
\end{align}
where they decoupled the circumcircle radius to
\begin{align*}
R^{(p,q)}(x,y,z) = \frac{(|y-z|\, |y-x| \, |z-x|)^p}{|(y-x)\wedge (z-x)|^q} = \frac{|y-z|^p |y-x|^{p-q} |z-x|^{p-q}}{ \sin \measuredangle (y-x,z-x)^q}
\end{align*} 
for any $p,q>0$. Observe that $\mathcal{M}^{p}(\gamma) = 2^p \textnormal{intM}^{(p,p)}(\gamma) $. According to \cite{BlattReiter:2015:Menger}, these energies are well-defined knot energies for $p\geq \tfrac 23 q + 1$ and give finite energy values on closed curves for $p< q + \tfrac 23$.  Moreover, minimizers exist within every knot class for $p\in (\tfrac 23 q +1,q+\tfrac 23)$ and $q>1$. 

For  the particular case $q=2$ and therefore $p\in (\tfrac 73, \tfrac 83)$, the generalized integral Menger curvature energies in \eqref{def:GenIntMengCurv} indeed lead to non-degenerate Euler-Lagrange equations, for which Blatt and Reiter proved the following regularity result by a combination of potential estimates and Sobolev-embeddings.

\begin{theorem}\label{thm:smooth}\cite[Theorem~4]{BlattReiter:2015:Menger}
	For $p\in (\tfrac 73, \tfrac 83)$, let $\gamma :\R/\Z\rightarrow \R^n$ be a simple curve in $W^{\frac 32 p -2,2}(\R/\Z,\R^n)$ parametrized by arc-length. 
	If $\gamma$ is a stationary point of the generalized integral Menger curvature  $\Mpz$ 
	with respect to fixed length,
	then $\gamma\in C^\infty(\R/\Z,\R^n)$. 
\end{theorem}

Furthermore, Blatt and Reiter characterized $C^1$-curves of finite generalized integral Menger curvature energy by curves in the fractional Sobolev space $W^{\frac {3p-2}{q} -1,q}(\R /\Z,\R^n)$ \cite[Theorem~1]{BlattReiter:2015:Menger}.
This result, combined with \thref{thm:smooth}, raises the question whether stationary points of finite generalized integral Menger curvature energies are not only smooth but also analytic. 

For the Möbius energy the statement was elaborated by \cite{BlattVorderobermeier:2019:CriticalMoebius} based on the corresponding smoothness result established by \cite{BlattReiterSchikorra:2016:CriticalPointsOfMoebiusEnergyAreSmooth}.  The same holds for a non-degenerate range of O'Hara's energies as shown in \cite{Vorderobermeier:2019:CriticalOHara} based on \cite{BlattReiter:2013:StationaryPointsOfOHarasKnotEnergies}. Hence it seems reasonable that it is possible to transfer the techniques to the non-degenerate range of energies $\Mpz$ for $p\in(\tfrac 73,\tfrac 83)$.

The main method used in these papers to show analyticity is based on the method of majorants, motivated by the proof of Cauchy-Kovalevsky's theorem \cite{Kovalevskaya:1875:ZurTheoriederpartiellenDifferentialgleichungen}, see also \cite[pp.~17--33]{Cauchy:1882:OEuvrescompletesdAugustinCauchy}. For applying this technique, it is necessary to establish a recursive estimate for derivatives of critical curves.
It turns out that in such estimates, there appears 
a fractional derivative of the product of two functions, which subsequently leads to a fractional Leibniz rule. Unfortunately, in contrast to the bilinear Hilbert transform in \cite{BlattVorderobermeier:2019:CriticalMoebius} and fractional Leibniz rule in \cite{Vorderobermeier:2019:CriticalOHara}, in our case this recursive estimate does not produce a sufficient difference in the order of differentiability, which is why we repeat it to obtain a difference of order $1$. 
This repetition seems to lead to Gevrey regularity of order $2$ for critical curves, which is weaker than analyticity, cf. \cite[Definiton 4.1.19]{KrantzParks:2002:Aprimerofrealanalyticfunctions} or \cite[Chapter~1.4]{Rodino:1993:LinearPartialDiffOperatorInGevreySpaces},
but we are able to confirm the analyticity in the affirmative:

\begin{theorem}\label{thm:main}
	For $p\in (\tfrac 73, \tfrac 83)$, let $\gamma :\R/\Z\rightarrow \R^n$ be a simple curve in $C^\infty(\R/\Z,\R^n)$ parametrized by arc-length. 
	If $\gamma$ is a critical point of the generalized integral Menger curvature  $\Mpz$ 
	with respect to fixed length,
	then the curve $\gamma$ is analytic. 
\end{theorem}

A closer look at the proof of the main theorem reveals that smoothness of $\gamma$ is actually not needed. It suffices to assume that $\norm[H^{\frac 5 2}]{\gamma^{(i)}}$ is finite for $i \in \set{0,1,2}$, i.e. $\gamma \in H^{4 + \frac 1 2}(\R/\Z,\R^n)$. All additional regularity is obtained by an inductive argument.

However, by the characterization of the energy spaces in \cite[Theorem~1]{BlattReiter:2015:Menger} and the regularity result \thref{thm:smooth}, finite energy $\Mpz$ of a critical point implies its smoothness already. We conclude:

\begin{corollary}\label{cor:main}
	Let  $\gamma :\R/\Z\rightarrow \R^n$ be a simple curve in $C^1(\R/\Z,\R^n)$ parametrized by arc-length  with $\Mpz (\gamma) < \infty$ for $p\in(\tfrac 73, \tfrac 83)$.
	If $\gamma$ is a critical point of $\Mpz$ with respect to fixed length, then the curve $\gamma$ is analytic.
\end{corollary}

One should be able to drop the $C^1$-assumption in the previous result by obtaining sufficient initial regularity via techniques that proved to be successful in showing a geometric Morrey-Sobolev imbedding for $\mathcal{M}^p$ with $p>3$, cf.~\cite[Theorem~1.2]{StrzeleckiSzumanskavonderMosel:2010:RegularizingandselfavoidanceeffectsofintegralMengercurvature}.

Further analyticity and Gevrey regularity results in the context of non-local differential equations can be found for example in \cite{DallAcquaEtAl:2012,DallAcquaEtAl:2014,AlbaneseEtAl:2015:GevreyRegularityForIntDiffOp,Blatt:2020:AnalyticityIntegroDiffOp,Blatt:2020:AnalyticityNonlinerEllipticPDE}.

\subsection*{Exposé of the present work}
As stated above, the central idea of our proof is the same as in the one for the Cauchy-Kovalevsky theorem and goes as follows.
A curve $\gamma \in C^\infty(\R/\Z,\R^n)$ is analytic if and only if we have constants $C>0$ and $r>0$ such that for all $l \in \N_0$
\[
	 \norm[L^\infty]{\gamma^{(l)}} \le C \frac {l!} {r^l}.
\]
Suppose we had a recursive estimate
\[
	\norm[L^\infty]{\gamma^{(l+1)}} \le \Phi(l;\norm[L^\infty]{\gamma},\ldots,\norm[L^\infty]{\gamma^{(l)}})
\]
and an analytic function $c:(-\epsilon, \epsilon) \to \R$ with
\[
	c^{(l+1)}(0) = \Phi(l;c(0),\ldots,c^{(l)}(0)) \text{ and } c(0) \ge \norm[L^\infty]{\gamma}.
\]
Then, by induction,
\[
	\begin{split}
		0 \le \norm[L^\infty]{\gamma^{(l+1)}} \le \Phi(l+1;c(0),\ldots,c^{(l)}(0)) = c^{(l+1)}(0) \le C \frac {(l+1)!} {r^{l+1}}
	\end{split}
\]
for all $l \in \N$ and so, $\gamma$ is analytic.

The largest part is devoted to obtaining the required recursive estimate.
In Section \ref{sec:Decomposition}, we review a decomposition of the first variation of $\Mpz$ given in \cite[Section~4]{BlattReiter:2015:Menger} into a highest- and lower-order terms called $\Qpt$ and $\Rpt$.
In a critical point, $\Qpt$ is equal to $-\Rpt$ and we may use the regularity gap between those to establish the recursive estimate.
To that end, Section \ref{sec:EstimateMainTerm} yields an estimate of $\partial^{l+3} \gamma$ by $\partial^l\Qpt(\gamma)$.
In Section \ref{sec:RemainderTermNonCritical} we bound $\partial^l \Rpt(\gamma)$ by $\partial^{l+2} \gamma$ via a fractional Leibniz rule.
Looking at the norms involved, it turns out that these estimates combined only lead to a regularity gain strictly between $\frac 1 2$ and $1$ which is not enough for the recursive estimate, so in Subsection \ref{subsec:RecursiveEstimate}, we iterate the above once more using the fact that $\gamma$ is a critical point.
The remainder of Section \ref{sec:MainProof} is dedicated to finding the analytic majorant $c$ which fulfils the same recursive estimate with equality.
Finally, Section \ref{sec:Consequences} discusses a few simple consequences of the analyticity of critical points, not only for $\Mpz$, in geometric knot theory.

We expect that our approach extends to functionals whose first variation admits a decomposition with properties similar to those of $\Qpt$ and $\Rpt$.
A probably valid generalization would be to assume only an arbitrarily small gain in regularity, see page \pageref{rem:smallerRegularityGain}.

\subsection*{Notation}\label{subsec:Notation}
In the following we give a short introduction to the fractional Sobolev spaces we are using in this paper. First, we present a continuation of the usual periodic Sobolev spaces of integer order $W^{k,p}(\R/\Z, \R^n)$ to Sobolev–Slobodeckij spaces $W^{k+s,p}(\R/\Z, \R^n)$ for $0<s <1$. The latter are defined by
\[
	W^{k+s,p}(\R/\Z, \R^n) := \{ f\in W^{k,p}(\R/\Z, \R^n) \ | \ [f^{(k)}]_{W^{s,p}(\R/\Z, \R^n)} < \infty \}
\]
equipped with the norm 
\[
	\|f\|_{W^{k+s,p}(\R/\Z, \R^n)} := \|f\|_W^{k,p}(\R/\Z, \R^n) + [f^{(k)}]_{W^{s,p}(\R/\Z, \R^n)},
\]
where 
\[
	[f^{(k)}]_{W^{s,p}(\R/\Z, \R^n)} := \Big(\int_{\R/\Z} \int_{-1/2}^{1/2} \frac{|f^{(k)}(x) - f^{(k)}(y)|^p}{|x-y|^{1+ps}} \d x \d y \Big)^{\frac 1p}
\]
is the so-called Gagliardo seminorm of $k$-th derivative of $f$. 

Second, we consider Bessel potential spaces of order $s \geq 0$ on the circle $\R/\Z$ that are given by 
\begin{align*}
H^s(\R/\Z, \C^n) := \{f\in L^2 (\R/\Z, \C^n) \ | \ \| f \|_{H^s}:= \| f \|_{H^s(\R/\Z,\C^n)} := \sqrt{(f,f)_{H^s}} < \infty  \},
\end{align*}
equipped with the inner product  
\begin{align*}
(f,g)_{H^s} := (f,g)_{H^s(\R/\Z,\C^n)} := \sum_{k\in\Z} (1+k^2)^s \langle \widehat{f}(k), \widehat{g}(k) \rangle_{\C^n}.
\end{align*}
Here, $\hat f(k) = \int_0^1 e^{-2 \pi i k x} f(x) \d x$ is the $k$-th Fourier coefficient of $f$.
For any integer $k\geq 0$, the introduced space $H^k$ coincides with the classical Sobolev space $W^{k,2}$  and their norms $\|\cdot\|_{H^k}$ and $\|f\|_{W^{k,2}}:= (\sum_{\nu=0}^{k} \|f^{(\nu)}\|^2_{L^2})^{\frac{1}{2}}$ are equivalent (cf.~\cite[Lem.~1.2]{Reiter:2012:RepulsiveKnotEnergies} or \cite[7.62]{AdamsFournier:2003:SobolevSpaces}).
Moreover, Bessel potential spaces coincide with Sobolev–Slobodeckij spaces in the case of $p=2$, in particular, they have equivalent norms
as stated in \cite[Proposition~3.4]{DiNezzaetal:2012:HitchhikersGuide} for spaces on the domain $\R^d$. The proof can be modified to fit our situation.
It can be shown that $H^s(\R/\Z,\R^n)\subseteq H^t(\R/ \Z,\R^n)$ for any $0 \leq t<s$ (cf.~\cite[Chapter~4,~Proposition~3.4]{Taylor:1996:PartialDifferentialEquationsa}) as well as $H^s(\R/\Z,\R^n) \subseteq C(\R/\Z,\R^n)$ for any $s> \frac{1}{2}$ (cf.~\cite[Chapter~4,~Proposition 3.3]{Taylor:1996:PartialDifferentialEquationsa}).

Furthermore, $H^m$ satisfies the Banach algebra property for any $m\geq \frac 12$, i.e. there exists a positive constant $C = C(m)$ such that 
\begin{align}\label{prop:Banachalgebra}
\|fg\|_{H^m} \leq C \|f\|_{H^m} \|g\|_{H^m}
\end{align}
for all $f,g \in H^m(\R/\Z,\R)$  (cf.~\cite[Theor.~4.39]{AdamsFournier:2003:SobolevSpaces}).

For further information on fractional Sobolev spaces, we refer the reader for instance to \cite{AdamsFournier:2003:SobolevSpaces,DiNezzaetal:2012:HitchhikersGuide,RunstSickel:1996:SobolevSpacesOfFractionalOrder, Triebel:2010:TheoryOfFunctionSpaces}.

\newcommand{\G}[1][p]{\ensuremath{\mathcal G^{(#1)}}}
\newcommand{\T}[3][\g]{\triangle_{#2,#3}{#1}}
\newcommand{\g}{\gamma}
\newcommand{\dg}{{\gamma}'}
\newcommand{\br}[1]{\ensuremath{\left(#1\right)}}
\newcommand{\Wia}[1][(3p-2)/q-1,q]{\ensuremath{W_{\mathrm{ia}}^{#1}}}
\newcommand{\s}{\ensuremath{\sigma}}
\newcommand{\gp}{g^{p}}
\newcommand{\W}[1][(3p-2)/q-1,q]{\ensuremath{W^{\scriptstyle #1}}}

\section{Decomposition of the first variation}
\label{sec:Decomposition}

Hermes was the first to derive Gateaux differentiability and a formula for the first variation of the integral Menger curvature $\mathcal{M}_p =  2^p \textnormal{intM}^{(p,p)} $ for $p\geq 2$ in his PhD thesis \cite[Theorem\,2.33, Remark\,2.35]{Hermes:2012:Menger}. In \cite[Theorem\,3]{BlattReiter:2015:Menger}, Blatt and Reiter extended this result to the generalized integral Menger curvature $\textnormal{intM}^{(p,q)}$ for the sub-critical range $p\in (\tfrac 23 q+1, q+\tfrac 23)$ and $q>1$. However, they approached the problem differently to Hermes by taking advantage of studying a certain subdomain of integration as well as using the newly discovered characterization of energy spaces developed in \cite[Theorem\,1.1]{Blatt:2013:NoteMenger} and \cite[Theorem\,1]{BlattReiter:2015:Menger}. The latter states that injective curves in $C^1(\R/\Z,\R^n)$ parametrized by arc-length have finite generalized integral Menger curvature iff they belong to the fractional Sobolev space $W^{(3p-2)/q-1,q} (\R/\Z, \R^n)$ \cite[Theorem 1]{BlattReiter:2015:Menger}. 

To summarize Blatt and Reiter's findings regarding the first variation of  the generalized integral Menger curvature $\textnormal{intM}^{(p,q)}$, cf. \cite[Theorem 3]{BlattReiter:2015:Menger}, they showed that it is continuously differentiable on the subspace of all regular embedded $W^{(3p-2)/q-1,q}$-curves. Furthermore, they gave an explicit formula for the first variation  of $\textnormal{intM}^{(p,q)}$ at  any arc-length parametrized embedded $\gamma \in W^{(3p-2)/q-1,q}(\R/\Z, \R^n)$ in direction $h \in W^{(3p-2)/q-1,q}(\R/\Z, \R^n)$.

In the following we will restrict ourselves to the non-degenerate, sub-critical range of generalized integral Menger curvature, i.e. $q=2$ and $p\in(\tfrac 73, \tfrac 83)$. In \cite{BlattReiter:2015:Menger} it has turned out to be helpful in that case to decompose its first variation  into a term of highest order and lower order terms in order to study the regularity of critical points.

Hence we start with recapitulating a decomposition of the first variation introduced in \cite[(4.2) and Lemma 4.2]{BlattReiter:2015:Menger}. 
Note that we abbreviate
\[
\difference{x,y} f  := f(u+x) - f(u+y) 
\]
and 
\[
D:= \set*[w \leq 1 + 2v, v \geq -1 + 2w]{(v,w) \in \left(-\tfrac 1 2,0\right) \times \left(0,\tfrac 1 2\right)}.
\]
\sloppy
We observe for any arc-length parametrized embedded $\gamma \in W^{\frac 32 p-2,2}(\R/\Z, \R^n)$ and $h \in W^{\frac 32 p-2,2}(\R/\Z, \R^n)$ that
\[
\label{eq:firstVariation}
\numberthis
\delta \Mpz (\gamma,h) = 12(\Qp(\gamma,h) + R^{(p)} (\gamma,h)),
\]
\fussy
where the main term is given by
\[
\Qp(\gamma,h) := \iiint_{\R/\Z \times D}  \frac {\left< \frac {\difference{w,0} \gamma} {w} - \frac {\difference{v,0} \gamma} {v}, \frac {\difference{w,0} h} {w} - \frac {\difference{v,0} h} {v} \right>} {\abs{v-w}^p \abs{v}^{p-2} \abs{w}^{p-2}} \d w \d v \d u.
\]

The remainder term can be expressed as following. 
\begin{lemma}[\protecting{\cite[Lemma 4.2]{BlattReiter:2015:Menger}}]
	\label{lemma:formRemainder}
	The term $R^{(p)} (\gamma,h)$ is a finite sum of terms of the form 
	\[
	\iiint\limits_{\mathbb R / \mathbb Z \times D} \idotsint\limits_{[0,1]^K} g^p(u,v,w;s_{1},\dots,s_{K-2}) \otimes h'(u+s_{K-1}v + s_K w) 
	\d\theta_{1}\cdots\d\theta_{K}\d v \d w \d u
	\numberthis
	\label{lemma:formRemainder:formula}
	\]
	where $\G : (0,\infty)^3 \rightarrow \mathbb R$ is an analytic function, $\otimes$ may denote any kind of product
	structure, such as cross product, dot product, scalar or matrix multiplication, $s_{j}\in\set{0,\theta_{j}}$ for $j=1,\dots,K$,
	\begin{multline*}
	g^p(u,v,w;s_{1},\dots,s_{K-2})=
	\G \left(\frac{|\T 0w|}{|w|}, \frac{|\T 0v|}{|v|}, \frac{|\T vw|}{|v-w|}\right) 
	\Gamma(u,v,w,s_{1},s_{2})\cdot{}\\
	{}\cdot\left( \bigotimes_{i=3}^{K_1} \dg(u+ s_i v) \right) \otimes \left(\bigotimes_{j=K_1}^{K_2} \dg(u + s_i w) \right) \otimes \left(\bigotimes_{j=K_2}^{K-2} \dg(u+v + s_i (w-v)) \right),
	\end{multline*}
	and
	$\Gamma(u,v,w,s_{1},s_{2})$ is a term of one of the four types
	\begin{align*}
	&\frac{\br{\gamma'(u+s_1 w) - \gamma'(u+s_1 v)}\otimes\br{\gamma'(u+s_2 w) - \gamma'(u+s_2 v)}}{|v|^{p-2} |w|^{p-2} |v-w|^{p}}, \\
	&\frac{|\gamma'(u+s_1 w) - \gamma'(u+s_2 w)|^2}{|v|^{p-2} |w|^{p-2} |v-w|^{p}}, \\ 
	& \frac{|\gamma'(u+s_1 v) - \gamma'(u+s_2 v)|^2}{|v|^{p-2} |w|^{p-2} |v-w|^{p}}, \\ 
	& \frac{|\gamma'(u+v + s_1 (w-v)) - \gamma'(u+v+s_2 (w-v))|^2}{|v|^{p-2} |w|^{p-2} |v-w|^{p}}.
	\end{align*}
\end{lemma}

\section{Estimate of the main term $Q^{(p)}$}
\label{sec:EstimateMainTerm}

Next we investigate the order of the main term $Q^{(p)}$, for which we need to examine its $L^2$-representation first.

\begin{lemma}
    \label{lemma:L2MainTermViaFourier}
    Let $p \in (\tfrac 7 3,\tfrac 8 3)$, $s \ge 0$ and $\gamma \in H^{3p-4+s}(\R/\Z,\R^n)$ be parametrized by arc-length.
    Then, there is $\Qpt(\gamma) \in H^s(\R/\Z,\R^n)$ such that for all $h \in H^{\frac 3 2 p -2}(\R/\Z,\R^n)$,
    \[
        \Qp(\gamma,h) = \int_{\R/\Z} \left<\Qpt(\gamma)(u), h(u)\right> \d u.
    \]
    In particular, $\Qpt$ maps $C^\infty(\R/\Z,\R^n)$ into itself.
    The Fourier coefficients of $\Qpt$, for all $k\in \Z$, are given by
    \begin{align*}\label{eq:fouriercoeffQ}
        \widehat{\Qpt(\gamma)}(k) = q_k^{(p)} \abs{k}^{3p-4} \hat{\gamma} (k).
    \end{align*}
    The constants $q_k^{(p)}$ are bounded independently of $n$, and satisfy 
    \[
        0 \leq q_k^{(p)} = c + o(1)
    \]
    as $\abs{k} \rightarrow \infty$ for $c>0$.
\end{lemma}
\begin{proof}
    By \cite[Proposition 4.1]{BlattReiter:2015:Menger}, we have
    \[
        \Qp(\gamma,h) = \sum_{k \in \Z}\rho_k \left< \hat \gamma(k),\hat h(k)\right>_{\C^n}, \text{ where } \rho_k = c\abs{k}^{3p-4} + o(\abs{k}^{3p-4}) \text{ as } \abs{k} \to \infty,
    \]
    for any $h \in H^{\frac 3 2 p -2}(\R/\Z,\R^n)$. In particular,  the $\rho_k$ are independent of $n$, cf.\ their definition in the proof of \cite[Proposition 4.1]{BlattReiter:2015:Menger}.
    Setting $q_k^{(p)} := \abs{k}^{4-3p} \rho_k$, which implies that the $q_k^{(p)}$ are also independent of $n$, we obtain
    \[
        \begin{split}
            (1+k^2)^{\frac s 2}\abs[\big]{q_k^{(p)}\abs{k}^{3p-4}\hat \gamma(k)}
            &= (1+k^2)^{\frac s 2}\abs{c+o(1)}\abs{k}^{3p-4} \abs{\hat \gamma(k)}\\
            &\le (1+k^2)^{\frac s 2} \tilde c \abs{k}^{3p-4} \abs{\hat \gamma(k)}
            \le \tilde c (1+k^2)^{\frac {3p-4+s} {2}} \abs{\hat \gamma(k)} \text{ as } \abs{k} \to \infty
        \end{split}
    \]
    for some $\tilde c>0$.
    The right hand side, taken as a sequence in $k \in \Z$, is in $\ell^2(\Z)$ because $\gamma \in H^{3p-4+s}(\R/\Z,\R^n)$, see \cite[Chapter~4, (3.8)]{Taylor:1996:PartialDifferentialEquationsa},
    and so $\Qpt(\gamma)$, the inverse Fourier transform of $q_k^{(p)}\abs{k}^{3p-4}\hat \gamma(k)$, is well-defined and in $H^s(\R/\Z,\R^n)$.
    All $q_k^{(p)}\abs{k}^{3p-4}$ are real as $\Qp$ maps real functions to real numbers and $\overline{\hat \gamma(k)} = \hat \gamma (-k)$ due to the real image of $\gamma$.
    Thus, $\Qpt(\gamma)(u) = \sum_{k\in \Z} q_k^{(p)}\abs{k}^{3p-4} \hat{\gamma}(k) e^{-2\pi i k u} = \hat \gamma(0) + \sum_{k \in \N} q_k^{(p)}\abs{k}^{3p-4} \Re(\hat \gamma(k) e^{-2 \pi i k u})$ for almost all $u \in \R/\Z$ meaning that $\Qpt(\gamma)$ maps into the real numbers.
    
    By Parseval's relation \cite[Proposition~3.2.7~(3)]{Grafakos:2014:ClassicalFourierAnalysis}, we have
    \[
        \int_{\R/\Z} \left<\Qpt(\gamma)(u),h(u)\right> \d u
        = \sum_{k\in \Z} \left< q_k^{(p)}\abs{k}^{3p-4}\hat \gamma(k), \hat h(k)\right>_{\C^n}
        = \sum_{k \in \Z}\rho_k \left< \hat \gamma(k),\hat h(k)\right>_{\C^n}
        = \Qp(\gamma,h).
    \]
\end{proof}
\begin{remark}
    Assuming $C^5$-regularity of $\gamma$, one may explicitly calculate
    \begin{align*}
        \Qpt(\gamma)(u) :=& \iint_D \abs{v-w}^{-p} \abs{v}^{2-p} \abs{w}^{2-p} \cdot \Biggl(
        \frac {\frac{\difference{-w,0} \gamma} {-w} - \frac {\difference{w,0} \gamma} {w}} {w}
        + \frac {\frac {\difference{-v,0} \gamma} {-v} - \frac {\difference{v,0} \gamma} {v}} {v}\\
        &+ \frac {\frac {\difference{v,0} \gamma} {v} - \frac {\difference{v-w,-w} \gamma} {v}} {w}
        + \frac {\frac {\difference{w,0} \gamma} {w} - \frac {\difference{w-v,-v} \gamma} {w}} {v}
        \Biggr)
        \d w \d v.
    \end{align*}
    The proof works via first restricting to the case $\abs{v},\abs{w}>\epsilon$ to obtain the formula and then showing existence for $\epsilon = 0$ by means of several Taylor approximations.
    Since this method is quite technical and long, we choose to omit it here.
\end{remark}

From the previous \thref{lemma:L2MainTermViaFourier}, we deduce the following indispensable corollary regarding the order of the main term $Q^{(p)}$.

\begin{corollary}  \label{cor:Ql}
	Given real numbers $p\in (\tfrac 73, \tfrac 83)$ and $m\geq 0$, there exists a constant $ \widetilde{C}=\widetilde{C}(p)>0$ such that for all curves $\gamma \in C^\infty(\R/\Z, \R^n)$ and  integers $l\geq 0$ 
	\begin{align*}
	\|\gamma^{(l+3)} \|_{H^{m+3p-7}} \leq \widetilde{C} \|\partial^{l} \Qpt(\gamma)\|_{H^m}
	\end{align*}
	holds.
\end{corollary}
\begin{proof}
	By applying \thref{lemma:L2MainTermViaFourier}, well-known properties of Fourier coefficients and the elementary estimate $(1+k^2)^{3p-7} \leq (2|k|)^{2(3p-7)}$ for any $k\in \Z  \setminus 0$,  we get
	\begin{align*}
	\|\partial^{l}\Qpt (\gamma) \|_{H^m}^2 &= \sum_{k\in \Z} (1+|k|^2)^m|\widehat{\partial^{l}\Qpt (\gamma)}(k)|^2 \\
	&= \sum_{k\in \Z} (1+|k|^2)^m (2\pi| k|)^{2l} |\widehat{\Qpt (\gamma)}(k)|^2 \\
	&= \sum_{k\in \Z} (1+|k|^2)^m (2\pi| k|)^{2l} (q_k^{(p)})^2 |k|^{2(3p-4)} |\hat{\gamma} (k)|^2\\
	&\geq \inf_{k\in\Z\setminus 0} \{(q_k^{(p)})^2 (2\pi)^{-6} 2^{2(7-3p)} \} \sum_{k \in \Z} (1+|k|^2)^{m+3p-7} (2\pi| k|)^{2(l+3)}   |\hat{\gamma} (k)|^2 \\
	&= \widetilde{C}^{-2} \sum_{k \in \Z} (1+|k|^2)^{m+3p-7} |\widehat{\gamma^{(l+3)} } (k)|^2 \\
	&= \widetilde{C}^{-2}   \|\gamma^{(l+3)} \|_{H^{m+3p-7}}^2,
	\end{align*}
	where $\widetilde{C}:=\inf_{k\in\Z\setminus{0}} \{\abs{q_k^{(p)}} (2\pi)^{-3} 2^{7-3p}\}^{-1}$ is a positive constant only depending on $p$.
    In particular, independence of $n$ follows from  \thref{lemma:L2MainTermViaFourier}.
\end{proof}

\section{The remainder term $R^{(p)}$}
\label{sec:RemainderTermNonCritical}

Now we aim to estimate the remainder term $R^{(p)}$, which appears in the first variation of the generalized integral Menger curvature for $q=2$ and $p\in (\tfrac 73, \tfrac 83)$, and its derivatives. We observe that it is indeed of lower order in comparison to the main term $Q^{(p)}$, which has been studied in the previous section.

\subsection{Representation of the remainder term $R^{(p)}$}

In this subsection we are concerned with an alternative representation of the remainder term $R^{(p)}$ and its derivatives.
Let us first recall the following result.

\begin{lemma}[Regularity of the remainder integrand, \protecting{\cite[Lemma 4.4]{BlattReiter:2015:Menger}}]
    \label{lemma:regularityIntegrandRemainder}\ \\
    Let $\g\in W^{(3p-4)/2+\s,2}(\R/\Z,\R^n)$ be a simple curve parametrized by arc-length.
    \begin{itemize}
        \item If $\s=0$, then $\gp\in L^{1}(\R/\Z\times D,\R^{n})$ and
        \item if $\s>0$, then $((v,w) \mapsto \gp(\cdot,v,w;\dots))\in L^{1}(D,\W[\tilde\s,1](\R/\Z, \R^{n}))$ for any $\tilde\s<\s$.
    \end{itemize}
    The respective norms are bounded independently of $s_{1},\dots,s_{K}$.
\end{lemma}

We then derive a representation of the remainder term $R^{(p)}$ corresponding to the $L^2$-re\-pre\-sen\-ta\-tion of the main term $Q^{(p)}$ in the first variation of $\Mpz$. Recall that  $\otimes$ denotes any kind of bilinear product structure, such as cross product, dot product, scalar or matrix multiplication.

\begin{lemma}
    \label{lemma:L2formRemainder}
    Let $\gamma \in C^\infty(\R/\Z,\R^n)$ be parametrized by arc-length and let $h \in W^{\frac 3 2 p -2,2}(\R/\Z,\R^n)$. 
    Then,
    \[
        R^{(p)}(\gamma,h)=\int_{\R/\Z} \left<\Rpt(\gamma)(u), h(u)\right>  \d u
    \]
    with
    \[	
        \Rpt(\gamma)_o(u) := \sum_{k=1}^{k_{\max}} \iint_{D} \idotsint\limits_{[0,1]^K} \frac \d {\d u} g^p_k(u - s_{K-1}v - s_K w,v,w;s_{1},\dots,s_{K-2}) \d\theta_{1}\cdots\d\theta_{K}\d v \d w \otimes e_o.
    \]
    Here, $e_o \in \R^n$ is the $o$-th canonical unit vector, each $g^p_k$ is of the same form as $g^p$ in \thref{lemma:formRemainder} and neither $K$ nor $k_{\max}$ depend on $\gamma$.
    Furthermore, $u \mapsto g^p_k(u,v,w;s_1,\ldots,s_K) \in C^\infty(\R/\Z,\R^n)$ for all $k \in \set{1,\ldots,k_{\max}}$.
\end{lemma}
\begin{proof}
	Since $\gamma \in C^\infty(\R/\Z,\R^n)$, \thref{lemma:regularityIntegrandRemainder} yields that the integrand is in fact integrable on the integration domain, allowing us to substitute $u$ by $\tilde u := u +s_{K-1}v + s_K w$.
	We also obtain that $u \mapsto g^p(u,v,w;s_1,\ldots,s_K) \in C^\infty(\R/\Z,\R^n)$ and note that $\otimes$ is a bilinear operation. Thus we may use integration by parts component-wise to obtain that \eqref{lemma:formRemainder:formula} is equal to
	\[
	\iiint\limits_{\mathbb R / \mathbb Z \times D} \idotsint\limits_{[0,1]^K} (\tfrac \d {\d \tilde u} g^p(\tilde u - s_{K-1}v - s_K w,v,w;s_{1},\dots,s_{K-2}) ) \otimes h(\tilde u)
	\d\theta_{1}\cdots\d\theta_{K}\d v \d w \d \tilde u.
	\]
	Note that we have no boundary terms because both $h$ and $g^p$ are periodic.
	Using the bilinearity of $\otimes$, we may write \Rpt{} as
    \[
        \int_{\R/\Z} \overline{R}^{(p)}(\gamma)(u) \otimes h(u) \d u 
    \]
    with
    \[
        \overline{R}^{(p)}(\gamma)(u) := \sum_{k=1}^{k_{\max}} \iint_{D} \idotsint\limits_{[0,1]^K} \frac \d {\d u} g^p_k(u - s_{K-1}v - s_K w,v,w;s_{1},\dots,s_{K-2}) \d\theta_{1}\cdots\d\theta_{K}\d v \d w.
    \]
    Seeing as $x \mapsto \overline{R}^{(p)}(\gamma)(u) \otimes x \in \R$ is linear, we may set
    \[
        \Rpt(\gamma)_o(u) := \overline{R}^{(p)}(\gamma)(u) \otimes e_o
    \]
    and obtain for $x = \sum_{o=1}^n x_oe_o \in \R^n$:
    \[
        \overline{R}^{(p)}(\gamma)(u) \otimes x = \sum_{o=1}^n x_o \cdot \overline{R}^{(p)}(\gamma)(u) \otimes e_o = \sum_{o=1}^n x_o \cdot \Rpt(\gamma)_o(u) = \left<\Rpt(\gamma)(u),x\right>.
    \]
\end{proof}

The $l$-th derivative of the remainder representation \Rpt{} can be computed as follows.
\begin{lemma}
    Let $\gamma \in C^\infty(\R/\Z,\R^n)$ be parametrized by arc-length.
    Then, 
    \[
        \begin{split}
            &\frac {\d^l} {\d u^l} \Rpt(\gamma)_o(u)\\
            =& \sum_{k=1}^{k_{\max}} \iint_{D} \idotsint\limits_{[0,1]^K} \frac {\d^{l+1}} {\d u^{l+1}} g^p_k(u - s_{K-1}v - s_K w,v,w;s_{1},\dots,s_{K-2}) \d\theta_{1}\cdots\d\theta_{K}\d v \d w\otimes e_o
        \end{split}
    \]
    for all $l \in \N$.
\end{lemma}
Note that we do not need the assumption of arc-length parametrization here, as $\Rpt$ is well-defined without it.
However, for other parametrizations, we do not expect $\Rpt(\gamma)$ to be related to $R^{(p)}(\gamma,\cdot)$.
\begin{proof}
    Since $X \mapsto X \otimes e_o$ is a fixed linear map, $\frac {\d^l} {\d u^l} (X(u) \otimes e_o) = ( \frac {\d^l} {\d u^l} \partial^l X(u)) \otimes e_o$, so we only need to worry about the integral term.
    By \thref{lemma:regularityIntegrandRemainder}, $((v,w) \mapsto \gp_k(\cdot,v,w;\dots))\in L^{1}(D,W^{\tilde \sigma,1}(\R/\Z, \R^{n}))$ for all $\tilde\sigma>0$ and its norm is independent of $s_1, \ldots, s_K$.
    This means that $g^p_k$ is differentiable with respect to $u$ arbitrarily many times and the derivatives are all integrable with the majorant
    \[
        \norm[L^\infty]{\partial_1^{l+1} g^p_k(\cdot,v,w;s_1,\ldots,s_K)}.
    \]
    By the Morrey- (see e.g.\ \cite[Paragraph~4.16]{AdamsFournier:2003:SobolevSpaces}) and Sobolev-embedding theorems, we can bound this above by
    \[
        C\norm[W^{1,2}]{\partial_1^{l+1} g^p_k(\cdot,v,w;s_1,\ldots,s_K)}
        \le \widetilde C \norm[W^{l + \frac 5 2,1}]{g^p_k(\cdot,v,w;s_1,\ldots,s_K)}.
    \]
    As we have a uniform $L^1$-bound on the latter, we may exchange integration and differentiation.
\end{proof}

\subsection{A fractional Leibniz rule} \label{subsec:FractionalLeibnizRule}

The remainder term $R^{(p)}$ contains amongst others the factor $\Gamma$, which is basically given as a product of two functions multiplied with a singular weight. By applying a Bessel potential space norm, we obtain a sort of fractional Leibniz rule which will subsequently be discussed for $\Gamma$ and its derivatives.

We observe that the $l$-th derivative of $R^{(p)}$ leads to $k$-th derivatives of $\Gamma$, where $k\leq l+1$. Applying the generalized Leibniz rule to the $ k$-th derivative of $\Gamma$ in the next step, we obtain
\[
\begin{split}
& \frac {\d^{k}} {\d u^{k}} \Gamma(u,v,w,s_1,s_2)  = \frac{1}{|v|^{p-2}|w|^{p-2}|v-w|^{p}} \\
& \cdot  \sum_{o=0}^{k} \binom{k}{o} \left(\gamma^{(k-o+1)}(u+x_1) - \gamma^{(k-o+1)}(u+x_2) \right) \otimes \left( \gamma^{(o+1)}(u+x_3) - \gamma^{(o+1)}(u+x_4) \right),
\end{split}
\]
where we have the four cases 
\[
\begin{split}
\textbf{Case 1:} & \quad (x_1,x_2,x_3,x_4) =(s_1w,s_1v, s_2w, s_2v), \\
\textbf{Case 2:} & \quad (x_1,x_2)=(x_3,x_4) = (s_1w,s_2w), \\
\textbf{Case 3:} & \quad (x_1,x_2)=(x_3,x_4) = (s_1v, s_2v), \\
\textbf{Case 4:} & \quad (x_1,x_2)=(x_3,x_4) = (v+s_1(w-v), v+s_2(w-v)).
\end{split}
\]
The key observation here is that $\Gamma$ as well as its derivatives can be broken down into sums of products of the form
\[
	\left(\gamma_i^{(k-o+1)}(u+x_1) - \gamma_i^{(k-o+1)}(u+x_2) \right) \cdot \left( \gamma_j^{(o+1)}(u+x_3) - \gamma_j^{(o+1)}(u+x_4) \right),
\]
where $i,j=1,\ldots,n$ and $o=0,\ldots, k$. The difference in the component functions and the additional factor $|v|^{2-p}|w|^{2-p}|v-w|^{-p}$ indicate a fractional derivative of the arising products. This motivates to deduce the following general statement, which can therefore be interpreted as fractional Leibniz rule.

\begin{theorem}\label{thm:fracLeibnizrule}
	Let $m> \frac{1}{2}$, $s_1,s_2 \in [0,1]$ and $p \in (\frac 73,\frac 83)$. Then there exists a positive constant $C_L = C_L(m, p) < \infty$  such that for all $f,g \in  C^\infty(\R/\Z,\R)$
	\begin{align*}\label{eq:fracLeibnizrule}
	& \iint_D   \frac{ \left\| \left(f(\cdot+x_1) - f(\cdot+x_2) \right) \left(g(\cdot+x_3) - g(\cdot+x_4) \right) \right\|_{H^m}}{|v|^{p-2}|w|^{p-2}|v-w|^{p}}  \d v \d w \nonumber  \\
	&  \leq C_L ( \|f\|_{H^{m+ \frac 32 p -3}} \|g\|_{H^{m+ \frac 32 p -3 }} ),
	\end{align*}
	where 
	\[
	\begin{split}
	\textbf{Case 1:} & \quad (x_1,x_2,x_3,x_4) =(s_1w,s_1v, s_2w, s_2v), \\
	\textbf{Case 2:} & \quad (x_1,x_2)=(x_3,x_4) = (s_1w,s_2w), \\
	\textbf{Case 3:} & \quad (x_1,x_2)=(x_3,x_4) = (s_1v, s_2v), \\
	\textbf{Case 4:} & \quad (x_1,x_2)=(x_3,x_4) = (v+s_1(w-v), v+s_2(w-v)).
	\end{split}
	\]
\end{theorem}
\begin{proof}
	We first of all note that $H^m$ is a Banach algebra due to $m>\frac 12$ (see \eqref{prop:Banachalgebra}), which implies 
	\begin{align*}
	& \iint_D   \frac{ \left\| \left(f(\cdot+x_1) - f(\cdot+x_2) \right) \left(g(\cdot+x_3) - g(\cdot+x_4) \right) \right\|_{H^m}}{|v|^{p-2}|w|^{p-2}|v-w|^{p}}  \d v \d w \nonumber  \\
	& \leq C(m) \iint_D   \frac{ \left\| f(\cdot+x_1) - f(\cdot+x_2) \right\|_{H^m} \left\|g(\cdot+x_3) - g(\cdot+x_4) \right\|_{H^m}}{|v|^{p-2}|w|^{p-2}|v-w|^{p}}  \d v \d w \nonumber  \\
	& \leq C(m) \left(\iint_D   \frac{ \left\| f(\cdot+x_1) - f(\cdot+x_2) \right\|^2_{H^m}}{|v|^{p-2}|w|^{p-2}|v-w|^{p}}  \d v \d w \right)^{\frac 12}  \left(\iint_D   \frac{\left\|g(\cdot+x_3) - g(\cdot+x_4) \right\|^2_{H^m}}{|v|^{p-2}|w|^{p-2}|v-w|^{p}}  \d v \d w \right)^{\frac 12}
	\end{align*}
	for some positive constant $C(m)$ only depending on $m$. 
	
	We start with \textbf{Case 1}. We first see by the definition of the fractional Sobolev norm that
	\begin{align*}
	& \iint_D   \frac{ \left\| f(\cdot+s_1w) - f(\cdot+s_1 v) \right\|^2_{H^m}}{|v|^{p-2}|w|^{p-2}|v-w|^{p}}  \d v \d w \\
	& = \iint_D  \sum_{k\in\Z} (1+k^2)^m  |(\widehat{f(\cdot+s_1w) - f(\cdot+s_1v)})(k)|^2   \frac{\d v \d w }{|v|^{p-2}|w|^{p-2}|v-w|^{p}} .
	\end{align*}
	By substitution and the periodicity of $f$, we can rewrite the Fourier coefficients to 
	\begin{align}\label{eq:FouriercoeffCOV}
	|\widehat{\left(f(\cdot+s_1w) - f(\cdot+s_1v)\right)}(k)| & =  \left|\int_0^1 (f(u+s_1w) -  f(u+s_1v)) e^{-2\pi i k u}\d u\right| \nonumber \\
	& =  \left|e^{2\pi i k s_1 w}\right| \, \left|\int_0^1 (f(u) -  f(u+s_1(v-w))) e^{-2\pi i k u}\d u\right| \nonumber \\
	& = |(\widehat{f(\cdot+s_1(v-w))-f(\cdot)})(k)|,
	\end{align}
	which implies 
	\[
	\left\| f(\cdot+s_1w) - f(\cdot+s_1 v) \right\|_{H^m} = \left\|  f(\cdot+s_1 (v-w)) - f(\cdot)\right\|_{H^m}.
	\]
	Then, as in the proof of \cite[Lemma 4.4]{BlattReiter:2015:Menger}, respectively \cite[Lemma 1]{BlattReiter:2015:Menger}, we substitute with $\Phi:(v,w) \mapsto (t,\tilde{w}):= (\frac{v}{v-w}, s_1(v-w))$, $|\det D\Phi(v,w)| = \frac{s_1}{|v-w|}$,  $\Phi(D)\subset [0,1]\times [-1,0]$, such that
	\begin{align*}
	&\iint_D   \frac{ \left\| f(\cdot+s_1w) - f(\cdot+s_1 v) \right\|^2_{H^m}}{|v|^{p-2}|w|^{p-2}|v-w|^{p}}  \d v \d w \\
    &  = \iint_D   \frac{ \left\|  f(\cdot+s_1 (v-w)) - f(\cdot)\right\|^2_{H^m}}{|v|^{p-2}|w|^{p-2}|v-w|^{p}}  \d v \d w \\
	& \leq  s_1^{-1} \int_0^1 \int_{-1}^0 \frac{ \left\|  f(\cdot+\tilde{w}) - f(\cdot)\right\|^2_{H^m}} {|t \frac{\tilde{w}}{s_1}|^{p-2}|(t-1)\frac{\tilde{w}}{s_1}|^{p-2}|\frac{\tilde{w}}{s_1}|^{p-1}}  \d \tilde{w} \d t \\
	& = s_1^{3p-6}  \left( \int_{-1}^{0} \frac{ \left\|  f(\cdot+\tilde{w}) - f(\cdot)\right\|^2_{H^m}}{|\tilde{w}|^{3p-5}} \d \tilde{w} \right) \smash{\underbrace{\left(\int_0^1 \frac{dt}{|t(1-t)|^{p-2}} \right)}_{< \infty}}\\
	& \leq C(p)  \int_{-1}^{0} \frac{ \left\|  f(\cdot+\tilde{w}) - f(\cdot)\right\|^2_{H^m}}{|\tilde{w}|^{3p-5}} \d \tilde{w}.
	\end{align*}
	Note that the positive constant $C(p)$ in the previous estimate does not depend on $s_1$ since $3p-6 >1$ and by definition  $0\leq s_1 \leq 1$.
	
	Now we find again by the definition of the fractional Sobolev norm 
	\begin{align*}
	\int_{-1}^{0} \frac{ \left\|  f(\cdot+\tilde{w}) - f(\cdot)\right\|^2_{H^m}}{|\tilde{w}|^{3p-5}} \d \tilde{w} = \int_{-1}^{0} \sum_{k\in\Z} (1+k^2)^m |(\widehat{f(\cdot+\tilde{w}) - f(\cdot)})(k)|^2 \frac{d \tilde{w}}{|\tilde{w}|^{3p-5}},
	\end{align*}
	whose Fourier coefficients have the form
	\begin{align*}
	\left|(\widehat{f(\cdot+\tilde{w}) - f(\cdot)})(k)\right|^2 & =  \left|\int_0^1 f(u+\tilde{w}) e^{-2\pi i k u} \d u  - \int_0^1 f(u) e^{-2\pi i k u} \d u\right|^2 \\
	& = |e^{2\pi i k \tilde{w}} - 1|^2  | \hat{f}(k)|^2. 
	\end{align*}
	Since we get by Euler's formula 
	\begin{align} \label{eq:orderk}
	\begin{split}
	\int_{-1}^{0} |e^{2\pi i k \tilde{w}} - 1|^2  \frac{d \tilde{w}}{|\tilde{w}|^{3p-5}} & = 2 \int_0^1 (1-\cos(2\pi k\tilde{w}) ) \, \frac{d \tilde{w}}{\tilde{w}^{3p-5}} \\
	& \leq 2 (2\pi |k|)^{3p-6} \underbrace{\int_0^\infty \frac{1-\cos(\bar{w})}{\bar{w}^{3p-5}} \d \bar{w}}_{<\infty} \\
	& \leq \widetilde C(p)\,  |k|^{3p-6}, 
	\end{split}
	\end{align}
	for some positive constant $\widetilde C(p)$ only depending on $p$, 
	as well as  the elementary inequality 
	\begin{align}\label{eq:elementariest}
	|k|^{3p-6} \leq (1+k^2)^{\frac{3p-6}{2}}, 
	\end{align}
	we obtain together with Beppo Levi's monotone convergence theorem
	\begin{align*}
	\int_{-1}^{0} \frac{ \left\|  f(\cdot+\tilde{w}) - f(\cdot)\right\|^2_{H^m}}{|\tilde{w}|^{3p-5}} \d \tilde{w} & = \int_{-1}^{0} \sum_{k\in\Z} (1+k^2)^m | \hat{f}(k)|^2   |e^{2\pi i k w} - 1|^2\frac{d \tilde{w}}{|w|^{3p-5}} \\
	& \leq \widetilde C(p)  \sum_{k\in\Z} (1+k^2)^{m+\frac 32 p -3} | \hat{f}(k)|^2 \\
	& = \widetilde C(p) \|f\|_{H^{m+\frac 32 p -3}}^2. 
	\end{align*}
	After applying the same reasoning for the second factor, we finally achieve for \textbf{Case 1}
	\begin{align*}
	& \iint_D   \frac{ \left\| \left(f(\cdot+s_1w) - f(\cdot+s_1v) \right) \left(g(\cdot+s_2w) - g(\cdot+s_2v) \right) \right\|_{H^m}}{|v|^{p-2}|w|^{p-2}|v-w|^{p}}  \d v \d w \\
	& \leq C(m,p) \|f\|_{H^{m+\frac 32 p -3}} \|g\|_{H^{m+\frac 32 p -3}},
	\end{align*}
	for some new positive constant $C(m,p)$ depending on $m$ and $p$.

	In \textbf{Case 2}, we start by change of variables, cf.\ the proof of \cite[Lemma 4.4]{BlattReiter:2015:Menger}, such that
	\begin{align}\label{eq:startcase2}
	& \iint_D   \frac{ \left\| f(\cdot+s_1w) - f(\cdot+s_2w) \right\|^2_{H^m}}{|v|^{p-2}|w|^{p-2}|v-w|^{p}}  \d v \d w \nonumber \\
	& \leq \int_{0}^{\frac 23}  \frac{ \left\| f(\cdot+s_1w) - f(\cdot+s_2w) \right\|^2_{H^m}}{w^{p-2}}  \left( \int_0^{\frac 23 - w} \frac{1}{v^{p-2}(v+w)^{p}} \d v \right) \d w \nonumber \\
	& \leq \int_{0}^{\frac 23}  \frac{ \left\| f(\cdot+s_1w) - f(\cdot+s_2w) \right\|^2_{H^m}}{w^{p-2}}  \left( \int_0^{\infty} \frac{1}{(tw)^{p-2}(w(t+1))^{p}} w \d t \right) \d w\nonumber\\
	& \leq \int_{0}^{\frac 23}  \frac{ \left\| f(\cdot+s_1w) - f(\cdot+s_2w) \right\|^2_{H^m}}{w^{3p-5}} \underbrace{ \left( \int_0^{\infty} \frac{1}{t^{p-2}(1+t)^{p}} \d t \right) }_{<\infty} \d w \nonumber \\
	& \leq C (p) \int_{0}^{\frac 23}  \frac{ \left\| f(\cdot+s_1w) - f(\cdot+s_2w) \right\|^2_{H^m}}{w^{3p-5}}  \d w, 
	\end{align}
	where $C(p)$ is a positive constant depending on $p$.
	The definition of the fractional Sobolev norm gives us 
	\begin{align*}
	\int_{0}^{\frac 23}  \frac{ \left\| f(\cdot+s_1w) - f(\cdot+s_2w) \right\|^2_{H^m}}{w^{3p-5}}  \d w = \int_{0}^{\frac 23} \sum_{k\in\Z} (1+k^2)^m  |(\widehat{f(\cdot+s_1w) - f(\cdot+s_2w)})(k)|^2  \frac{dw}{w^{3p-5}}.
	\end{align*}
	We then observe for the Fourier coefficients of the series, recalling of the periodicity of $f$,
	\begin{align*}
	|(\widehat{f(\cdot+s_1w) - f(\cdot+s_2w)})(k)|^2 & =  \left|\int_0^1 f(u+s_1w) e^{-2\pi i k u}\d u - \int_0^1  f(u+s_2w)e^{-2\pi i k u}\d u\right|^2 \\
	& =  |e^{2\pi i k s_1 w} \hat{f}(k) - e^{2\pi i k s_2 w} \hat{f}(k)|^2 \\
	& = |e^{2\pi i k (s_1 -s_2)w}-1|^2 |e^{2\pi i k s_2 w}|^2 |\hat{f}(k)|^2, 
	\end{align*}
	and hence 
	\begin{align*}
	\int_{0}^{\frac 23}  \frac{ \left\| f(\cdot+s_1w) - f(\cdot+s_2w) \right\|^2_{H^m}}{w^{3p-5}}  \d w = \int_{0}^{\frac 23} \sum_{k\in\Z} (1+k^2)^m |\hat{f}(k)|^2   |e^{2\pi i k (s_1 -s_2)w}-1|^2 \frac{dw}{w^{3p-5}}.
	\end{align*}
	Now again by Euler's formula and change of variables, we get similarly to \eqref{eq:orderk}
	\begin{align*}
	\int_{0}^{\frac 23}  |e^{2\pi i k (s_1 -s_2)w}-1|^2 \frac{dw}{w^{3p-5}} & = 2 \int_{0}^{\frac 23}  1-\cos (2\pi  k (s_1 -s_2)w) \frac{dw}{w^{3p-5}} \leq \widetilde C(p) |k|^{3p-6}
	\end{align*}
    where $\widetilde C(p)$ is a positive constant depending on $p$ and we used that $\abs{s_1-s_2} \le 2$.
	Therefore, we obtain by \eqref{eq:elementariest} and Beppo Levi's monotone convergence theorem
	\begin{align}\label{eq:case2}
	\int_{0}^{\frac 23}  \frac{ \left\| f(\cdot+s_1w) - f(\cdot+s_2w) \right\|^2_{H^m}}{w^{3p-5}}  \d w & \leq \widetilde C(p)  \sum_{k\in\Z} (1+k^2)^{m+\frac 32 p -3} |\hat{f}(k)|^2  \nonumber\\
	& = \widetilde C (p) \|f\|^2_{H^{m+\frac 32 p -3}}. 
	\end{align}
	After applying the same arguments for the factor of $g$, we conclude in \textbf{Case 2}
	\begin{align*}
	& \iint_D   \frac{ \left\| \left(f(\cdot+s_1w) - f(\cdot+s_2w) \right) \left(g(\cdot+s_1w) - g(\cdot+s_2w) \right) \right\|_{H^m}}{|v|^{p-2}|w|^{p-2}|v-w|^{p}}  \d v \d w \\
	& \leq C(m,p) \|f\|_{H^{m+\frac 32 p -3}} \|g\|_{H^{m+\frac 32 p -3}}
	\end{align*}
	for some positive constant $C(m,p)$ depending on $m$ and $p$.

	\textbf{Case 3} directly follows from \textbf{Case 2}. \\
	Now only \textbf{Case 4} remains. Along the lines of \textbf{Case 1} and \textbf{Case 2}, in particular \eqref{eq:FouriercoeffCOV} and \eqref{eq:startcase2}, we transform with respect of the Fourier coefficients 
	\begin{align*}
	& \iint_D   \frac{ \left\| f(\cdot+v+s_1(w-v)) - f(\cdot+v+s_2(w-v)) \right\|^2_{H^m}}{|v|^{p-2}|w|^{p-2}|v-w|^{p}}  \d v \d w \\
	& \le \int_0^{2/3} \, \int_0^{2/3-w} \frac{ \left\| f(\cdot+s_1(w+v)) - f(\cdot+s_2(w+v)) \right\|^2_{H^m}}{|w|^{p-2}} \left(  \frac{1}{v^{p-2}(v+w)^p} \right) \d v  \d w \\
	& \leq \int_{0}^{2/3}  \frac{ \left\| f(\cdot+s_1\tilde{w}) - f(\cdot+s_2\tilde{w}) \right\|^2_{H^m}}{\tilde{w}^{3p-5}}  \underbrace{\left( \int_0^{\infty} \frac{1}{t^{p-2}(1+t)^{p}} \d t \right)}_{<\infty} \d \tilde{w}.
	\end{align*}
	We are now in the position to apply \eqref{eq:case2} and the statement for \textbf{Case 4}
	\begin{align*}
	& \iint_D   \left\| \left(f(\cdot+v + s_1(w-v)) - f(\cdot+v + s_1(w-v)) \right) \cdot\right.\\
    & \qquad \left.\cdot \left(g(\cdot+v + s_1(w-v)) - g(\cdot+v + s_1(w-v)) \right) \right\|_{H^m} \frac {\d v \d w} {|v|^{p-2}|w|^{p-2}|v-w|^{p}}  \\
	& \leq C(m,p) \|f\|_{H^{m+\frac 32 p -3}} \|g\|_{H^{m+\frac 32 p -3}}
	\end{align*}
	for some positive constant $C(m,p)$ only dependent on $m$ and $p$
	follows. 
\end{proof}

\subsection{Estimate of the remainder term \texorpdfstring{$\Rpt$}{R(p)}} 

We may find an estimate for of $\partial^l \Rpt(\gamma)$ with respect to a fractional Sobolev norm.
\begin{lemma}
    \label{lemma:estimateRemainder}
    Let $\gamma \in C^\infty(\R/\Z,\R^n)$ be parametrized by arc-length and $m > \tfrac 12$.
    Then there exist constants $C=C(\gamma,m,n,p)>0$ and
     $r= r(\gamma,m,n,p)>0$  such that
    \begin{align*}
    	& \norm[H^m]{\partial^l \Rpt(\gamma)} \\
    	& \leq C \sum_{j_1 + \ldots + j_{K-1} = l+1} \binom{l+1}{j_1, \ldots, j_{K-1}} \ p^{(3n)}_{j_1} \left( \left\{\frac{(|\alpha|+1)!}{r^{|\alpha|+1}} \right\}_{|\alpha|\leq j_1}, \{ \| \gamma^{(j+1)} \|_{H^{m+ \frac 32 p -3}}\}_{\substack{j =1, \ldots, j_1\\ i = 1, \ldots, 3n}} \right) \\
    	& \phantom{\le} \cdot \prod_{i=2}^{K-1} {\|\gamma^{(j_i+1)}\|_{H^{m+ \frac 32 p -3}}}.
    \end{align*}
    Note that the second set of coefficients of $p^{(3n)}_{j_1}$ is independent of $i$, we need that parameter only for dimensional reasons.
\end{lemma}

\begin{proof}
	First note that $\Gp \circ |.|$ is an analytic function away from zero and by the fundamental theorem of calculus 
	\[
	\begin{split}
		&\Gp\left(\frac {|\difference{v,w} \gamma|} {|v-w|}, \frac {|\difference{v,0} \gamma|} {|v|}, \frac {|\difference{w,0} \gamma|} {|w|} \right)\\
        &= (\Gp \circ |.|) \left(\int_0^1 \gamma'(u+tw)\d t, \int_0^1 \gamma'(u+tv)\d t, \int_0^1 \gamma'(u+t(w-v))\d t \right).
	\end{split}
	\]
	Now by the generalized Leibniz formula and Faà di Bruno's formula as in \thref{lem:FaadiBruno}, we observe
	\[
	\begin{split}
	& \frac {\d^{l+1}} {\d u^{l+1}} g_k^p(u,v,w;s_1, \ldots, s_{K-2})
	\\
	& = \sum_{k_1 + \ldots + k_{K-2} = l+1} \binom{l+1}{k_1, \ldots, k_{K-2}}
	\\
	&\cdot p^{(3n)}_{k_1}\Biggl(
	\set*{ \partial^\alpha (\Gp \circ |. |) \left(\int_0^1 \gamma'(u+tw)\d t, \int_0^1 \gamma'(u+tv)\d t, \int_0^1 \gamma'(u+t(w-v))\d t \right)}_{\abs{\alpha}\leq k_1}, 
	\\
	& \qquad 
	\left\{ \int_0^1 \gamma^{(j+1)}_i (u+t(x-y)) \d t \right\}_{\substack{i=1,\ldots,n, \, j =1, \ldots, k_1 ,\\ (x,y) = (v,w), (v,0), (w,0)}} \Biggr)
	\\
	& \cdot  \frac {\d^{k_2}} {\d u^{k_2}} \Gamma(u,v,w,s_1,s_2)  \\
	& \cdot  \left( \bigotimes_{i=3}^{K_1} \gamma^{(k_i + 1)}(u+ s_i v) \right) \otimes \left(\bigotimes_{i=K_1 +1}^{K_2} \gamma^{(k_i + 1)}(u + s_i w) \right) \otimes \left(\bigotimes_{i=K_2 +1}^{K-2} \gamma^{(k_i + 1)}(u+v + s_i (w-v)) \right).
	\end{split}
	\]
	We then apply the $H^m$-norm on $\partial^l \Rpt(\gamma)$ for some $m> \tfrac 12$ and obtain by the Banach algebra property of $H^m (\R/\Z, \R^n)$, cf. \eqref{prop:Banachalgebra},  together with \thref{lemma:estimateFaaDiBrunoBanachAlgebra} that 
	\[
    \numberthis
    \label{lemma:estimateRemainder:eq:firstStep}
	\begin{split}
	& \left\| \partial^l \Rpt(\gamma) \right\|_{H^m}
	\\
	& \leq c(m,n,p) \sum_{k=1}^{k_{\max}} \iint_{D} \idotsint\limits_{[0,1]^K} \left\| \partial^{l+1} g_k^p(\cdot,v,w;s_1, \ldots, s_{K-2})\right\|_{H^m} \d\theta_{1}\cdots\d\theta_{K}\d v \d w
	\\
	& \leq \tilde c(m,n,p,K)   \sum_{k=1}^{k_{\max}} \iint_{D} \idotsint\limits_{[0,1]^K}  \sum_{k_1 + \ldots + k_{K-2} = l+1} \binom{l+1}{k_1, \ldots, k_{K-2}}
	\\
	&\cdot p^{(3 n)}_{k_1}\Biggl(\!
	\set*{b^{\abs{\alpha}} \norm*[H^m]{\partial^\alpha (\Gp \circ |.|) \left(\int_0^1 \gamma'(\cdot+tw)\d t, \int_0^1 \gamma'(\cdot+tv)\d t, \int_0^1 \gamma'(\cdot+t(w-v))\d t  \right)}}_{\hspace{-0.4em}\abs{\alpha}\leq k_1}\hspace{-2em},
	\\
	& \qquad 
	\left\{ \left\| \int_0^1 \gamma^{(j+1)}_i (\cdot+t(x-y)) \d t \right\|_{H^m}\right\}_{\substack{i=1,\ldots,n, \, j =1, \ldots, k_1 ,\\ (x,y) = (v,w), (v,0), (w,0)}} \Biggr)
	\\
	& \cdot \left\|\partial^{k_2}\Gamma(\cdot,v,w,s_1,s_2) \right\|_{H^m}
	\cdot \prod_{i=3}^{K-2} \norm*[H^m]{\gamma^{(k_i+1)}} \d\theta_{1}\cdots\d\theta_{K}\d v \d w
	\end{split}
	\]
    where $b = b(m) := \sup_{\eta_1,\eta_2 \in H^m(\R/\Z,\R)} \frac {\norm[H^m]{\eta_1}\norm[H^m]{\eta_2}} {\norm[H^m]{\eta_1\eta_2}}$ is a finite constant stemming from the Banach algebra property  and $c(n,m,p)$ is an upper bound for the operator norms of $\cdot \otimes e_o : H^m(\R/\Z,\R^n) \to H^m(\R/\Z,\R^n)$.
    
	Since $G^{(p)}\circ |.|$ is an analytic function away from the origin and $\gamma$ is smooth with bounded bilipschitz constant \cite[Proposition~2.1]{BlattReiter:2015:Menger},
	 we get from the proof of \cite[Lemma~7.2]{BlattVorderobermeier:2019:CriticalMoebius} that 
	$$
	b^{\abs{\alpha}}\norm*[H^m]{\partial^\alpha (\Gp \circ |.|) \left(\int_0^1 \gamma'(\cdot+tw)\d t, \int_0^1 \gamma'(\cdot+tv)\d t, \int_0^1 \gamma'(\cdot+t(w-v))\d t  \right)} 
	\leq C_{G}  \frac{(|\alpha|+1)!}{r^{|\alpha|+1}}
	$$
	for some positive constants $C_{G} = C_{G} (\gamma, m, n, p)$ and $r= r(\gamma,m,n,p)$ which also absorb $b$. Furthermore, we have by the triangle inequality
	\[
		\left\| \int_0^1 \gamma^{(j+1)}_i (\cdot+t(x-y)) \d t \right\|_{H^m} \leq \| \gamma^{(j+1)}_i \|_{H^m} \leq  \| \gamma^{(j+1)} \|_{H^m}.
	\]
	Note that these bounds are independent of $s_3, \ldots, s_{K-2}$.
	
	To estimate the term $\iint_D \iint_{[0,1]^2} \left\| \partial^{k_2} \Gamma(\cdot,v,w,s_1,s_2) \right\|_{H^m} \d \theta_1 \d \theta_2 \d w \d v$, we first observe by the generalized Leibniz rule
	\[
	\begin{split}
		& \frac {\d^{k_2}} {\d u^{k_2}} \Gamma(u,v,w,s_1,s_2)  = \frac{1}{|v|^{p-2}|w|^{p-2}|v-w|^{p}} \\
		& \cdot \sum_{o=0}^{k_2} \binom{k_2}{o} \left(\gamma^{(k_2-o+1)}(u+x_1) - \gamma^{(k_2-o+1)}(u+x_2) \right) \otimes \left( \gamma^{(o+1)}(u+x_3) - \gamma^{(o+1)}(u+x_4) \right),
	\end{split}
	\]
	where we have the four cases 
	\[
	\begin{split}
	\textbf{Case 1:} & \quad (x_1,x_2,x_3,x_4) =(s_1w,s_1v, s_2w, s_2v), \\
	\textbf{Case 2:} & \quad (x_1,x_2)=(x_3,x_4) = (s_1w,s_2w), \\
	\textbf{Case 3:} & \quad (x_1,x_2)=(x_3,x_4) = (s_1v, s_2v), \\
	\textbf{Case 4:} & \quad (x_1,x_2)=(x_3,x_4) = (v+s_1(w-v), v+s_2(w-v)).
	\end{split}
	\]
	Hence we get by \thref{thm:fracLeibnizrule} 
	\[
	\begin{split}
	&\iint_{[0,1]^2} \iint_D  \left\| \partial^{k_2} \Gamma(\cdot,v,w,s_1,s_2) \right\|_{H^m}  \d w \d v \d \theta_1 \d \theta_2  \\
	& \leq C(m,n,p) \sum_{i=1}^n \sum_{j=1}^n \sum_{o=0}^{k_2} \binom{k_2}{o} \iint_{[0,1]^2} \iint_D  \frac{1}{|v|^{p-2}|w|^{p-2}|v-w|^{p}} \\
	& \cdot \norm[H^m]{\gamma_i^{(k_2-o+1)}(\cdot +x_1) - \gamma_i^{(k_2-o+1)}(\cdot +x_2)} \norm[H^m]{\gamma_j^{(o+1)}(\cdot+x_3) - \gamma_j^{(o+1)}(\cdot +x_4)} \d w \d v \d \theta_1 \d \theta_2\\
	& \leq \widetilde C(m,n,p) \sum_{i=1}^n \sum_{j=1}^n \sum_{o=0}^{k_2} \binom{k_2}{o} \iint_{[0,1]^2} \| \gamma_i^{(k_2-o+1)}\|_{H^{m+ \frac 32 p -3}} \|\gamma_j^{(o+1)}\|_{H^{m+ \frac 32 p -3 }}  \d \theta_1 \d \theta_2 \\
	& \leq \widetilde C(m,n,p) \sum_{o=0}^{k_2} \binom{k_2}{o} \| \gamma^{(k_2-o+1)}\|_{H^{m+ \frac 32 p -3}} \|\gamma^{(o+1)}\|_{H^{m+ \frac 32 p -3 }}
	\end{split}
	\]
    where $C(m,n,p)$ is an upper bound for the operator norms of $e_i\otimes e_j$ on $(H^m(\R/\Z,\R))^2$.
    We may combine the sum over $o$ and the binomial coefficient $\binom{k_2}{o}$ with the sum over $k_1, \ldots, k_{K-2}$ and the multinomial coefficient $\binom{l+1}{k_1,\ldots,k_{K-2}}$ from \eqref{lemma:estimateRemainder:eq:firstStep} into the sum over $j_1 + \ldots + j_{K-1}$ and the multinomial coefficient $\binom{l+1}{j_1, \ldots, j_{K-1}}$ by setting $j_2 := o$, $j_{K-1} := k_2-o$ and $j_i := k_i$ for all other $i$.
    This allows us to integrate the two derivatives of $\gamma$ into the bigger product.
    
    Then, the
	statement follows from the embedding $H^s(\R/\Z,\C^n) \subseteq H^t(\R/\Z,\C^n)$ for any $s>t$, $s,t \geq 0$. 
    Since these estimates are uniform over all $g_k$, we may use an additional factor of $k_{\max}$ in order to leave out the sum.
\end{proof}
Note that $\Gamma$ is the only term for whose upper bound we need the higher-order $H^{m+\frac 3 2 p -3}$-norms.
We estimated all $H^m$-norms by these in order to get a simpler structure which will serve us in the proof of the main theorem, especially in \thref{lemma:iteratedEstimateMainTerm}.

\section{Proof of the main statement}
\label{sec:MainProof}

We are finally in the position to derive the key ingredient, an iterated recursive estimate for critical points $\gamma$ of the generalized integral Menger curvature $\Mpz$ with respect to a Sobolev norm. Based on that, we prove the main result of this paper.

\subsection{Recursive estimate and iteration}
\label{subsec:RecursiveEstimate}

In order to show analyticity of a curve, we need to bound its higher derivatives. One possible approach is to bound derivatives by the lower order ones and iterate this procedure inductively. In the following we show how such recursive estimates can be obtained from the Euler-Lagrange equations. 

Let  $\gamma\in C^\infty(\R/\Z,\R^n)$ be  a critical point of $\Mpz$ with respect to fixed length and $p\in(\tfrac 73,\tfrac 83)$.
By the Lagrange multiplier theorem (see e.g.\ \cite[§9.3~Theorem 1]{Luenberger:1998:OptimizationVectorSpace}), we have the existence of $\lambda \in \R$ such that
\[
0 = \delta\Mpz(\gamma,h) + \lambda \delta \LL(\gamma,h).
\]
Seeing as by integration by parts and periodicity, we have
\[
\delta\LL(\gamma,h) = \int_{\R/\Z} \left<\gamma'(u),h'(u)\right> \d u = - \int_{\R/\Z} \left<\gamma''(u),h(u)\right> \d u
\]
for $\gamma$ parametrized by arc-length.
Combining the decomposition of the first variation of $\Mpz$ mentioned in \eqref{eq:firstVariation} with the $L^2$-forms of the main term $Q^{(p)}$ in \thref{lemma:L2MainTermViaFourier} and the remainder term $R^{(p)}$ in \thref{lemma:L2formRemainder}, we may write
\[
0 =  \int_{\R/\Z} \left<12 \bigl(\Qpt(\gamma)(u) + \Rpt(\gamma)(u)\bigr) - \lambda \gamma''(u), h(u)\right> \d u.
\]
Thus, we have
\[
\label{eq:RelationMainAndRemainderTerms}
\numberthis
\Qpt(\gamma) = \frac {\lambda} {12} \gamma'' - \Rpt(\gamma).
\]
From this, we may easily deduce the following.
\begin{corollary}
	\label{cor:estimateMainTermViaRemainder}
	Let $\gamma \in C^\infty(\R/\Z,\R^n)$ be a critical point of $\Mpz$ with respect to fixed length and parametrized by arc-length.
	For $m>\frac 1 2$ there exist constants 
	$C=C(\gamma,m,n,p)>0$, $r=r(\gamma,m,n,p)>0$, and $\mu=\mu(\gamma, n,p)$ such that
	\[
	\begin{split}
	&\norm[H^{m+3p-7}]{\gamma^{(l+3)}}\\
	&\le C \sum_{j_1 + \ldots + j_{K-1} = l+1} \binom{l+1}{j_1, \ldots, j_{K-1}} \ p^{(3n)}_{j_1} \left( \left\{\frac{(|\alpha|+1)!}{r^{|\alpha|+1}} \right\}_{|\alpha|\leq j_1}, \{ \| \gamma^{(j+1)} \|_{H^{m+ \frac 32 p -3}}\}_{\substack{j =1, \ldots, j_1\\ i = 1, \ldots, 3n}} \right) \\
	& \phantom{\le} \cdot \prod_{i=2}^{K-1} {\|\gamma^{(j_i+1)}\|_{H^{m+ \frac 32 p -3}}} + \mu \norm[H^m]{\gamma^{(l+2)}}.
	\end{split}
	\]
\end{corollary}
\begin{proof}
	By \thref{cor:Ql}, we may estimate $\norm[H^{m+3p-7}]{\gamma^{(l+3)}} \le \widetilde C \norm[H^m]{\partial^l\Qpt(\gamma)}$.
	Noting that $\widetilde C$ only depends on $p$ and $n$, we only have to combine this estimate with \eqref{eq:RelationMainAndRemainderTerms} and \thref{lemma:estimateRemainder}. 
\end{proof}

The obtained estimate is already recursive, however, the difference in the order of differentiability is less than 1 since
\[
	3p-7 < \tfrac 32 p -3
\]
for any $p\in (\tfrac 73, \tfrac 83)$. This can lead to substantial problems in the iteration of the recursive estimate to bound higher derivatives. Therefore, we iterate the recursive estimate beforehand to get a difference of order at least 1 in the recursion.

To make the structure of the estimate more visible, set
\[
\numberthis
\label{eq:definitionPhi} 
\begin{split}
&\Phi(l,n,r,K;x_0,\ldots,x_l) :=\\
&\sum_{j_1 + \ldots + j_{K-1} = l} \binom{l}{j_1,\ldots,j_{K-1}} p_{j_1}^{(3n)}\left( \left\{ \frac {(\abs{\alpha} +1)!} {r^{\abs{\alpha} +1} } \right\}_{\abs{\alpha} \leq j_1}, \left\{ x_j \right\}_{\substack{j=1,\ldots,j_1\\ i=1,\ldots, 3n}} \right) \cdot \prod_{i=2}^{K-1} x_{j_i}
\end{split}
\]
and note that $\Phi$ is non-decreasing in the $x_i$.
The upper bound from \thref{cor:estimateMainTermViaRemainder} may be written as 
\[
\numberthis
\label{eq:upperBoundMainTermPhi}
C \Phi\left(l+1,n,r,K;\norm[H^{m+\frac 3 2 p -3}]{\gamma'}, \ldots, \norm[H^{m+\frac 3 2 p -3}]{\gamma^{(l+2)}}\right) + \mu \norm[H^m]{\gamma^{(l+2)}}.
\]

\begin{lemma}
	\label{lemma:iteratedEstimateMainTerm}
	Let $\gamma \in C^\infty(\R/\Z,\R^n)$ parametrized by arc-length be a critical point of $\Mpz$ with respect to fixed length and $m\geq \frac 5 2$.
	Then, there exist constants $\widehat C = \widehat C>0$, $\mu=\mu$ and $r>0$ all depending on $\gamma,m,n$ and $p$ 
	such that
	\[
	\begin{split}
	&\norm[H^{m+3p-7}]{\gamma^{(l+3)}}\\
	& \le C \Phi\Bigl( l+1, n, r,K;\widehat{C}\Phi\left[1,n,\frac {r} {2\pi},K; \norm[H^m]{\gamma}, \norm[H^m]{\gamma'}\right] + \mu \norm[H^m]{\gamma'}, \ldots,\\
	&\adjustedalignment{C \Phi\Bigl( l+1, n, r,K;\widehat{C}}{\widehat{C}}\Phi\left[l+2,n,\frac {r} {2\pi},K; \norm[H^m]{\gamma}, \ldots, \norm[H^m]{\gamma^{(l+2)}}\right] + \mu \norm[H^m]{\gamma^{(l+2)}}\Bigr)\\
	&+\mu \norm[H^m]{\gamma^{(l+2)}}.
	\end{split}
	\]
\end{lemma}
\begin{proof}
	Using \thref{cor:estimateMainTermViaRemainder} and \eqref{eq:upperBoundMainTermPhi}, it suffices to bound $\norm[H^{m + \frac 3 2 p -3}]{\gamma^{(\tilde l +1)}}$ for $\tilde l \in \{0, \ldots, l+1\}$.
	Since $m\ge \frac 52$, we obtain $m + \frac 3 2 p -5 >\frac 1 2$ and may use \thref{lemma:equivalentSobolevNormsDerivative}, \thref{cor:Ql} and \thref{cor:estimateMainTermViaRemainder} to estimate
	\[  
	\begin{split}
	\norm[H^{m + \frac 3 2 p -3}]{\gamma^{(\tilde l +1)}}
	&\le \norm[H^{m + \frac 3 2 p -5}]{\gamma^{(\tilde l +3)}} \\
	&\le \widetilde C(p) C \Phi\left({\tilde l}+1,n,r,K;\norm[H^{m-1}]{\gamma'}, \ldots, \norm[H^{m-1}]{\gamma^{({\tilde l}+2)}}\right) + \mu \norm[H^{m + 2 - \frac 3 2 p}]{\gamma^{({\tilde l}+2)}}.
	\end{split}    
	\]
	It holds that $m+2-\frac 3 2 p < m- \frac 3 2 < m-1$ and by \thref{lemma:equivalentSobolevNormsDerivative}, $\norm[H^{m-1}]{\gamma^{(j+1)}} \le 2 \pi \norm[H^m]{\gamma^{(j)}}$.
	A look at the definition of $\Phi$ together with \thref{lemma:estimateAnalyticFaaDiBruno} enables us to estimate the above by
	\[
	\widetilde C(p) C (2\pi)^{K-2} \Phi\left({\tilde l}+1,n,\frac {r} {2\pi},K;\norm[H^{m}]{\gamma}, \ldots, \norm[H^m]{\gamma^{({\tilde l}+1)}}\right) + \tilde \mu \norm[H^{m}]{\gamma^{({\tilde l}+1)}}
	\]
	for $\tilde \mu := \tilde\mu(\gamma,m,n,p)>0$.
	Relabeling $\mu := \max\set{\mu,\tilde \mu}$ gives the desired statement.
\end{proof}

\subsection{Proof of Theorem \ref{thm:main}} 

Now we are ready to proof the main statement of this paper by the method of majorants.

\begin{proof}[Proof of \thref{thm:main}]
	Recall we have given a simple closed curve
    \[
        \gamma = (\gamma_1,\ldots, \gamma_n) \in C^\infty(\R/\Z,\R^n)
    \]
	parametrized by arc-length which is a critical point of integral Menger curvature $\Mpz + \lambda \mathcal{L}$, subject to a fixed length constraint, where $p\in (\tfrac 73, \tfrac 83)$ and $\lambda \in \R$.
    We define 
    \[
        a_l := \norm[H^\mf]{\gamma^{(l)}}
    \]
	for all integers $l\geq 0$ by using the smoothness of the curve $\gamma$. 
	Our aim is show that there exist positive constants $C_\gamma$ and $r_\gamma$ such that we have for any integer $l\geq 0$
	\begin{align*}
	   a_l \leq C_\gamma \frac{l!}{r_\gamma^l};
	\end{align*}
	from which the analyticity of the critical point $\gamma$ on $\R/\Z$ immediately follows by \thref{cor:TayloranalyticSob}.

	First we establish a recursive estimate for the terms $a_l$. We begin with applying the criticality of the curve $\gamma$ and
    \thref{lemma:iteratedEstimateMainTerm} to obtain
    \[
        \label{eq:recursiveEstimateA}
        \numberthis
        \begin{split}
        	a_{l+3} \leq & \norm[H^{\mf + 3p - 7}]{ \gamma^{(l+3)} }\\
            \leq & C \Phi\Bigl( l+1, n, r,K;\widehat{C}\Phi\left[1,n,\frac {r} {2\pi},K; \norm[H^\mf]{\gamma}, \norm[H^\mf]{\gamma'}\right] + \mu \norm[H^\mf]{\gamma'}, \ldots,\\
            &\adjustedalignment{C \Phi\Bigl( l+1, n, r,K;\widehat{C}}{\widehat{C}}\Phi\left[l+2,n,\frac {r} {2\pi},K; \norm[H^\mf]{\gamma}, \ldots, \norm[H^\mf]{\gamma^{(l+2)}}\right] + \mu \norm[H^\mf]{\gamma^{(l+2)}}\Bigr)\\
            &+  \mu \norm[H^\mf]{\gamma^{(l+2)}}\\
            = & C \Phi\Bigl( l+1, n, r,K;\widehat{C}\Phi\left[1,n,\frac {r} {2\pi},K; a_0, a_1\right] + \mu a_1, \ldots,\\
            &\adjustedalignment{C \Phi\Bigl( l+1, n, r,K;\widehat{C}}{\widehat{C}}\Phi\left[l+2,n,\frac {r} {2\pi},K; a_0, \ldots, a_{l+2}\right] + \mu  a_{l+2}\Bigr) + \mu a_{l+2}.
    	\end{split}
    \]
	for positive constants $C, \widehat C, r, \mu$, all $l\ge 0$ and some $K \in \N$.
    
    Next, we try to find $\tilde a_l$ that satisfy this recursive formula with equality and such that the first three coefficients $\tilde a_0, \tilde a_1$ and $\tilde a_2$ are greater than or equal to $a_0,a_1$ and $a_2$, respectively.
    
    Define
    \begin{alignat*}{2}
        G: \R^{3n} \supseteq B_{\tilde \epsilon}(0) &\to \R^{3n},\\
        G_i(z) &:= \widehat C \left( 1 + \frac {3n a_0 - \sum_{k=1}^{3n} z_k} {\frac {r} {2\pi}} \right)^{-2} z_i^{K-2} + \mu z_i,\\
        F: \R^{3n} \supseteq B_{\tilde \epsilon}(0) &\to \R^{3n},\\
        F_i(y) &:= C \left( 1 + \frac {3n DG_1((a_0,\ldots,a_0)^\transpose)(a_1, \ldots,a_1)^\transpose  - \sum_{k=1}^{3n} y_k} {r} \right)^{-2} y_i^{K-2},\\
    \end{alignat*}
    for $\tilde \epsilon := \min\{3n a_0 + \frac {r} {2\pi}, 3n DG_1((a_0,\ldots,a_0)^\transpose)(a_1, \ldots,a_1)^\transpose + r\}$.
    These functions are clearly analytic and by the Cauchy-Kovalevsky \thref{thm:CK}, there is an analytic $c: (-\epsilon,\epsilon) \to \R^{3n}$ satisfying 
    \[
    \numberthis
    \label{eq:majorantODE}
    \left\{
        \begin{split}
            c''(t) &= \overline{C} \cdot \bigl(F(DG(c(t)) c'(t)) + \mu c'(t)\bigr),\\
            c(0) &= (a_0,\ldots,a_0)^\transpose,\\
            c'(0) &= (a_1,\ldots,a_1)^\transpose,
        \end{split}
    \right.
    \]
    where 
    \[
        \overline{C} := \max\left\{1, \frac{a_2} {F_1(DG((a_0,\ldots,a_0)^\transpose)(a_1,\ldots,a_1)^\transpose)} \right\}.
    \]
    A little further along, we give a formula for all derivatives of $G(c(t))$ which implies that $\overline C$ is well-defined.
    
    Since all entries of the right hand side in \eqref{eq:majorantODE} are identical given the initial values, we have that $c_i(t) = c_1(t)$ for all valid $i$ and $t$.
    Setting $\tilde a_l := c_1^{(l)}(0)$, we have by definition that 
    \[
        \label{eq:tildeAFirstThreeTerms}
        \numberthis
        \tilde a_l \ge a_l \text{ for } l=0,1,2.
    \]
    It remains to show the recursion.
    
    An application of Leibniz' rule and \eqref{universalpoly} yields
    \[
        \label{eq:derviativeGc}
        \numberthis
        \begin{split}
            &\restrict{\frac {\d^l} {\d t^l} G_i(c(t))}{t=0}\\
            = &\widehat C \sum_{j_1+ \ldots +j_{K-1}=l} \binom{l}{j_1, \ldots, j_{K-1}} p_l^{(3n)} \Biggl( \Biggl\{ \Biggl( 1+ 
            \underbrace{\frac {3n a_0 - \sum_{j=1}^{3n} c_j(0)} {\frac {r} {2\pi}}}_{=0}
            \Biggr)^{-\abs{\alpha} - 2} \frac {(\abs{\alpha} +1)!} {\left(\frac {r} {2\pi}\right)^{\abs{\alpha}+1}} \Biggr\}_{\abs{\alpha} \le j_1},\\
            & \phantom{\widehat C \sum_{j_1+ \ldots +j_{K-1}=l} \binom{l}{j_1, \ldots, j_{K-1}} p_l^{(3n)} \Biggl(} \left\{ c_j^{(k)}(0) \right\}_{\substack{k=1,\ldots,j_1\\j=1,\ldots,3n}} \Biggr)
            \cdot \prod_{k=2}^{K-1} c_i^{(j_k)}(0) + \mu c_i^{(l)}(0)\\
            = & \widehat C \Phi\left(l,n,\frac {r} {2\pi},K;c_1(0),\ldots,c_1^{(l)}(0)\right) + \mu c_1^{(l)}(0)
            =\widehat C \Phi\left(l,n,\frac {r} {2\pi},K;\tilde a_0,\ldots, \tilde a_l\right) + \mu \tilde a_l.
        \end{split}
    \]
    Note that this corresponds to the inner $\Phi$ terms in \eqref{eq:recursiveEstimateA}.
    
    For $l=1$ this is a polynomial in $\tilde a_0,\tilde a_1$ with positive coefficients, so as long as both $\tilde a_0 = a_0$ and $\tilde a_1 = a_1$ are positive, so is each component of $DG((a_0,\ldots,a_0)^\transpose)(a_1,\ldots,a_1)^\transpose$ and consequently $F_1$ applied to that.
    Since $\gamma$ is not constant, neither $\gamma$ nor $\gamma'$ are $0 \in H^{\mf}(\R/\Z,\R^n)$ and so $a_0$ and $a_1$ have to be positive, making $\overline C$ well-defined.
    
    Analogously to \eqref{eq:derviativeGc}, taking advantage of the identity for all components in $G$,
    \[
        \label{eq:derivativeRHS}
        \begin{split}
            &\tilde a_{l+3}
            = c_1^{(l+3)}(0)
            = \overline C \restrict{\frac {\d^{l+1}} {\d t^{l+1}} F_1\left(\frac {\d}{\d t} G(c(t))\right)}{t=0} + \overline C\mu c^{(l+2)}_1(0)\\
            =& \overline C C \sum_{j_1 + \ldots + j_{K-1} = l+1} \binom{l+1}{j_1, \ldots, j_{K-1}}\\
            &\quad p_l^{(3n)} \Biggl( \Biggl\{ \Biggl( 1+ 
            \underbrace{\frac {3n DG_1(c(0))\dot c(0) - \sum_{j=1}^{3n} \restrict{\frac {\d} {\d t} G_j(c(t))}{t=0}} {r}}_{=0}
            \Biggr)^{-\abs{\alpha} - 2} \frac {(\abs{\alpha} +1)!} {r^{\abs{\alpha}+1}} \Biggr\}_{\abs{\alpha} \le j_1},\\
            &\phantom{\quad p_l^{(3n)} \Biggl(}\left\{ \restrict{\frac{\d^{k+1}} {\d t^{k+1}} G_j(c(t))}{t=0} \right\}_{\substack{k=1,\ldots,j_1\\j=1,\ldots,3n}} \Biggr) \cdot \prod_{k=2}^{K-1} \restrict{\frac {\d^{j_k+1}}{\d t^{j_k+1}} G_1(c(t))}{t=0} + \overline{C} \mu \tilde a_{l+2}\\
            =&\overline{C} C \Phi\left(l+1,n,r,K; \restrict{\frac {\d}{\d t} G_1(c(t))}{t=0},\ldots,\restrict{\frac {\d^{l+2}}{\d t^{l+2}} G_1(c(t))}{t=0}\right)   + \overline{C} \mu \tilde a_{l+2}\\
            \ge& C \Phi\left(l+1,n,r,K; \restrict{\frac {\d}{\d t} G_1(c(t))}{t=0},\ldots,\restrict{\frac {\d^{l+2}}{\d t^{l+2}} G_1(c(t))}{t=0}\right)  + \mu \tilde a_{l+2},
        \end{split}
    \]
    hence we conclude together with \eqref{eq:derviativeGc} that
    \[
    \numberthis
    \label{eq:recursiveequation}
    \begin{split}
    & \tilde a_{l+3} \geq  C \Phi\Bigl( l+1, n, r,K;\widehat{C}\Phi\left[1,n,\frac {r} {2\pi},K; \tilde a_0, \tilde a_1\right] + \mu \tilde  a_1, \ldots,\\
    &\adjustedalignment{C \Phi\Bigl( l+1, n, r,K;\widehat{C}}{\widehat{C}}\Phi\left[l+2,n,\frac {r} {2\pi},K; \tilde a_0, \ldots, \tilde a_{l+2}\right] + \mu \tilde a_{l+2}\Bigr) + \mu \tilde a_{l+2}.
    \end{split}
    \]
    Using the fact that $\Phi$ is nondecreasing in $x_0,\ldots,x_l$ (see \eqref{eq:definitionPhi})
    , we obtain
    by comparing \eqref{eq:recursiveEstimateA} and  \eqref{eq:recursiveequation} inductively with initial values \eqref {eq:tildeAFirstThreeTerms}, that
    \[  
            \tilde a_{l} \geq a_{l}
    \]
    for all $l \in \N_0$.
    By the analyticity of $c_1$ in zero and \thref{cor:TayloranalyticSob},
    we have that there exist positive constants $C_\gamma$ and $r_\gamma$ such that 
    \[
        \norm[H^\mf]{ \gamma^{(l)}} = a_l \le \tilde a_l = c_1^{(l)} (0) \le C_\gamma \frac {l!} {r_\gamma^l}
    \]
    for all $l\in\N_0$ and thus, $\gamma$ is analytic.
\end{proof}
\label{rem:smallerRegularityGain}
Note that the method presented here ought to work with even smaller differences between the orders of differentiability of the main and the remainder terms. If the regularity gain in one step is $\frac 1 m$ we would expect an analogue of \eqref{eq:recursiveEstimateA} to hold with roughly $m$ layers of nested $\Phi$-terms attainable via an analogue of \thref{lemma:iteratedEstimateMainTerm}.
This approach would possibly lead to an $m$-th order ODE in place of \eqref{eq:majorantODE}.

\section{A few simple consequences concerning critical knots}
\label{sec:Consequences}

The main result of the paper states that critical points of the generalized integral Menger curvature $\Mpz$ are not only smooth, which was known before, cf. \cite[Theorem\,4]{BlattReiter:2015:Menger}, but also analytic. Recalling and employing special properties for analytic functions, we hence decude a few simple corollaries for critical knots of $\Mpz$ and some other knot energies.

The following statements hold for any functional $\EL: C^0(\R/\Z,\R^3) \rightarrow (0,\infty]$ whose (arc-length parametrized) critical points with respect to fixed length are analytic. As a consequence, the results also apply to the generalized integral Menger curvature $\Mpz$ and some of O'Hara's knot energies, cf. \cite{BlattVorderobermeier:2019:CriticalMoebius,Vorderobermeier:2019:CriticalOHara}.

Central to these consequences is a real version of the identity theorem. It is a well-known result for analytic functions which easily carries over to the periodic case:
\begin{proposition}
\label{propositionRealIdentityTheorem}
    If two analytic functions on $\R/\Z$ are identical at infinitely many distinct points, they are identical everywhere.
\end{proposition}
\begin{proof}
    Let $\gamma,\eta : \R/\Z \to \R^3$ be analytic and without loss of generality $t_1<t_2< \ldots \in [0,1)$ such that $\gamma(t_i)=\eta(t_i)$ for all $i \in \N$.
    Define $g,h: (-1,2) \to \R^3$ as the restrictions of $\gamma$ and $\eta$ to $(-1,2)$ which are clearly also analytic.
    As $t_i \in [0,1]$ for all $i \in \N$, we have a subsequence converging to $t_0 \in [0,1]\subseteq(-1,2)$, so by \cite[Corollary~1.2.7]{KrantzParks:2002:Aprimerofrealanalyticfunctions}, $g=h$ and thus $\gamma = \eta$.
\end{proof}

\begin{lemma}
    \label{lemma:zerosAnalyticFunctions}
    Let $\gamma \colon \R/\Z \to \R^3$ and $F: \R^3 \to \R$ both be analytic. Then, $F\circ \gamma$  either has finitely many roots or is uniformly zero.
\end{lemma}
\begin{proof}
    $F \circ \gamma \colon \R/\Z \to \R$ is analytic as a composition of analytic functions (see e.g. \cite[Proposition~2.2.8]{KrantzParks:2002:Aprimerofrealanalyticfunctions}).
    Using the zero function in \thref{propositionRealIdentityTheorem}, $F\circ \gamma$ either is uniformly zero or has only finitely many roots.
\end{proof}

\begin{proposition}
	\label{corollary:UnknotEllipsoidHyperplane}
	Let $\gamma \in C^1(\R/\Z,\R^3)$ be a simple curve parametrized by arc-length which is a critical point of $\EL$.
	Then, it is the unknot or it has only finitely many intersections with any hyperplane or sphere	in $\R^3$.
	In particular, $\gamma$ cannot have straight segments.
\end{proposition}
\begin{proof}
	First of all, let us consider any hyperplane with normal vector $n\in\R^3$ and position vector $x_0\in\R^3$, described by the equation $F(x)=0$, where  $F:\R^3\rightarrow\R$, $F(x)= \langle n,x-x_0\rangle$. Furthermore, let $\gamma:\R/\Z\rightarrow \R^3$ be a critical point of $\EL$ that is simple and parametrized by arc-length. We then observe that both $F$ and $\gamma$ are analytic, which is why we can apply \thref{lemma:zerosAnalyticFunctions} and observe that the curve $\gamma$ either intersects the hyperplane finitely many times or is entirely contained therein. In the latter case we obtain the unknot, which is the only simple planar knot, cf. \cite[p.~5]{CrowellFox:1977:Introductiontoknottheory} or \cite[Chapter~VI,~18.]{Newman:1951:ElementsOfTopology}.
    If $\gamma$ contains a straight segment, this segment is in particular contained in two distinct hyperplanes.
    Thus, the whole curve has to be contained in both hyperplanes and therefore has to be a straight line which is impossible for a closed $C^1$-curve of finite length.
	
	We then turn our focus to any sphere with center $x_0\in\R^3$ and radius $r>0$, described by the equation $F(x) =0$, where $F:\R^3 \rightarrow \R$, $F(x) = |x-x_0|^2 -r^2$.
    Therefore, similarly to the hyperplane, \thref{lemma:zerosAnalyticFunctions} gives us that any critical analytic knot $\gamma$  of $\EL$ either intersects the sphere finitely many times or is entirely contained therein. 

	In the second case, we would intuitively expect to obtain the unknot by applying a stereographic projection from the sphere to the plane and then argue as above. To make this rigorous, we need to construct an ambient isotopy, cf. \cite[Chapter~8.1]{Hirsch:1976:DifferentialTopology}, to avoid any kind of self-intersection or pull-tight phenomena during the projection process. 
	To simplify the upcoming computations, we assume that the sphere is the unit sphere centered in the origin and that the knot $\gamma$ on the sphere does not intersect the north pole of the sphere. We omit the proof that translations and dilations of the sphere do not change the knot type of the spherical knot and remark that stereographic projections can be done from any other point on the sphere as well. The existence of such point on the sphere, which the knot $\gamma$ does not intersect, is ensured due to the fact that $\gamma$ is rectifiable. Furthermore, note that since the curve $\gamma$ is continuous, there exists an $r>0$ such that $\gamma$ does not intersect a ball around the north pole with radius $r$.
	
	We begin with constructing a $C^1$-isotopy from $\gamma$ to the stereographic projection of $\gamma$ from the north pole onto the plane $z=0$ denoted by $\tilde \gamma$. Note that $\tilde \gamma$ is indeed a $C^1$-embedding: As $\gamma\in C^\omega(\R/\Z,\R^3)$ is a diffeomorphism onto its image and the stereographic projection $P:\mathbb S^2\setminus{(0,0,1)} \rightarrow \R^2$ given by $P(x,y,z)= (\frac{x}{1-z},\frac{y}{1-z})$ with inverse function $P^{-1}(u,v) = \tfrac{1}{u^2+v^2 +1}(2u,2v,u^2+v^2 -1)$ is a diffeomorphism from the sphere without the north pole to the plane $z=0$, also $\tilde \gamma \in C^1(\R/\Z,\R^3)$ is diffeomorphic onto its image. We then define the map $h:\R/\Z \times [0,1]\rightarrow \R^3$ by
	\[
		h_t(w):=(1-t)\gamma(w) + t \tilde \gamma (w),
	\]
    which is continuously differentiable in time and space.
    
    To see that $h$ is a $C^1$-isotopy (cf.\ \cite[Section~8.1]{Hirsch:1976:DifferentialTopology}), we construct a continuously differentiable left-inverse of each $h_t$ and thus show that $h_t$ is a $C^1$-embedding.
    First note that by the nature of stereographic projection, $h_t(w)$ lies on the unique straight line passing through both the north pole $N$ and $\gamma(w)$.
    Since $\gamma$ lies on the sphere and does not contain $N$, this line does not contain any other points of $\gamma$ or $h(\tilde w,[0,1])$ for any $\tilde w \ne w$.
    This means that for any $x \in h_t(\R/\Z)$, we may obtain $w=h_t^{-1}(x)$ as $\gamma^{-1}(y)$, where $y=\lambda N + (1-\lambda)x$ with $\abs{y}=1$.
    Solving for $\lambda \in \R$ yields either $\lambda=0$, which we may safely ignore since it corresponds to $y$ being the north pole, or $\lambda=\frac {\abs{x}^2-1} {\abs{x-N}^2}$.
    As $x$ has to stay away from the north pole, we found a $C^1$ right-inverse of $h_t$ and so, $h$ is a $C^1$-embedding.
    
    Obviously, also 
	\[
		h(\cdot,0) = \gamma \quad \textnormal{and} \quad h(\cdot,1) = \tilde \gamma
	\]
	hold. Now by \cite[Section~8.1,~Exercise~4]{Hirsch:1976:DifferentialTopology} or \cite[Theorem~1.2]{Blatt:2009:NoteOnIsotopies}, we can extend the $C^1$-isotopy between $\gamma$ and $\tilde \gamma$ to an ambient isotopy of $\R^3$ and therefore obtain that the knot type is preserved throughout the stereographic projection.
\end{proof}

As a direct consequence of \thref{lemma:zerosAnalyticFunctions} and  \thref{corollary:UnknotEllipsoidHyperplane}, we obtain the following. 
\begin{corollary}
	Let $\gamma \in C^1(\R/\Z,\R^3)$ be a simple curve parametrized by arc-length which is a critical point of $\EL$ and let $F:\R^3 \rightarrow \R$ be an analytic map that describes an implicit surface $F(x,y,z)=0$.
	Then, the curve $\gamma$ either is entirely contained in the implicit surface $F(x,y,z)=0$ or has only finitely many intersections with it. In particular, if the given surface is ambient isotopic to a sphere or a hyperplane, the curve $\gamma$ is the unknot or has finitely many intersections with the surface.
\end{corollary}

\begin{corollary}
    Let $\gamma,\eta$ be critical points of $\EL$ which are parametrized by arc-length.
    As soon as
    \[
        \gamma(t)=A\eta(t) + x_0
    \]
    for fixed $A\in \R^{3\times 3}$, $x_0 \in \R^3$, and infinitely many distinct $t \in [0,1)$, we have $\gamma = A \eta + x_0$.
\end{corollary}
\begin{proof}
    Note that by assumption, $\gamma$ and $\eta$ are analytic.
    Then, the statement follows from the fact that $A\eta + x_0$ is analytic as well and applying \thref{propositionRealIdentityTheorem}.
\end{proof}
\section*{Acknowledgements}

Daniel Steenebrügge acknowledges support by the Deutsche Forschungsgemeinschaft (DFG, German Research Foundation) – project number 320021702/GRK2326 –  \emph{Energy, Entropy, and Dissipative Dynamics (EDDy)} and Nicole Vorderobermeier by the Austrian Science Fund (FWF), Grant P29487. Funding from RWTH Aachen University for Nicole Vorderobermeier's research visit is also gratefully acknowledged.

In addition, the authors would like to sincerely thank Simon Blatt for his constant support and indispensable discussions throughout the development of this research paper, Heiko von der Mosel for his generous support and hospitality as well as fruitful discussions, from which Section \thref{sec:Consequences} emerged, and Philipp Reiter for having a sympathetic ear for the authors and giving important hints to deal with some filthy estimates in Section \thref{subsec:FractionalLeibnizRule}.

    \appendix

\section{Properties of analytic functions}

In this section we aim to recall some statements related to analytic functions, which are used in the paper. For further information on analytic functions we refer the reader for instance to \cite{KrantzParks:2002:Aprimerofrealanalyticfunctions}. 
First we note that analytic functions can be characterized as follows, cf.\ for example \cite[Corollary~3.2]{Vorderobermeier:2019:CriticalOHara}.

\begin{proposition} \label{cor:TayloranalyticSob}
	Let $f\in C^\infty (\mathbb R / \mathbb Z,\R^n)$ and $s>0$. Then the function $f$ is analytic on $\R/\Z$ if  there are positive constants $r$ and $C$ such that
	\begin{align*}
	\|f^{(k)}\|_{H^{1+s}} \leq C \frac{k!}{r^{k}}
	\end{align*}
	holds for all integers $k \geq 0$.
\end{proposition}

This can be seen by using \cite[Proposition~1.2.12]{KrantzParks:2002:Aprimerofrealanalyticfunctions} together with a standard covering argument, Morrey's inequality (see e.g. \cite[Paragraph~4.16]{AdamsFournier:2003:SobolevSpaces}), the equivalence of the $W^{1,2}$- and the $H^1$-norm on $\mathbb R / \mathbb Z$ (cf.~\cite[Lem.~1.2]{Reiter:2012:RepulsiveKnotEnergies} or \cite[7.62]{AdamsFournier:2003:SobolevSpaces}) as well as the embedding $H^s\subseteq H^t$ for any $t < s$ (cf.~\cite[Chapter~4,~Proposition~3.4]{Taylor:1996:PartialDifferentialEquationsa}).

In addition, we need a special case of the Cauchy-Kovalevsky theorem, which can be derived for example from \cite[(1.25)]{Folland:1995}.

\begin{theorem}[Cauchy-Kovalevsky -- ODE case]\label{thm:CK}
	Suppose the function $G: \R^{2n} \rightarrow \R^n$ is real analytic around $(c_0,c_1)$ for some $c_0,c_1 \in\R^n$, and the function $f \in C^\infty((-\varepsilon, \varepsilon),\R^n)$ of the form $f(x)= (f_1(x),\ldots,f_n(x))$, for any $ \varepsilon > 0$, is a solution of the initial value problem 
	\begin{align*}
	f''(x) & = G(f(x), f'(x)) \textnormal{ for } x \in (-\varepsilon,\varepsilon) \textnormal{ with}\\
	f(0)& =c_0, \textnormal{ and}\\
	f'(0) & = c_1. 
	\end{align*}
	Then the function $f$ is real analytic around $0$.
\end{theorem}

The statement can be proven by the \emph{method of majorants}, e.g.\ as explained in the proof of  \cite[2.4.1]{KrantzParks:2002:Aprimerofrealanalyticfunctions}. At this point we remark that the strategy of proof for the main statement, \thref{thm:main}, is motivated from this method.

\section{Faà di Bruno's formula}

Faà di Bruno's formula generalizes the chain rule to higher derivatives. In particular, the $k$-th derivative of the composition of two functions $f, g\in C^k (\R,\R)$ can be written as
\begin{align*}
&\left(\tfrac{d}{dt}\right)^k g(f(t)) = \\
&\sum_{\substack{m_1+2m_2+ \cdots km_k=k, \\ m_1,\ldots,m_k\in\N_0}} \frac{k!}{m_1!1!^{m_1}m_2!2!^{m_2}\ldots m_k!k!^{m_k}} g^{(m_1+\cdots +m_k)}(f(t))\prod_{j=1}^{k} (f^{(j)} (t))^{m_j}.
\end{align*}
R. Mishkov generalized Faà di Bruno's formula to the multivariate case \cite{Mishkov:2000:GeneralizationFormulaFaa}.
The following lemma and notation via the universal polynomial are taken directly from \cite[Section~2.3]{BlattVorderobermeier:2019:CriticalMoebius}.
\begin{lemma}\label{lem:FaadiBruno}
	Let $U\subseteq\R$ be a neighborhood of $0$ and $f\in C^k(U, \R^n)$, $n\in\N_0$. Moreover, let $V\subseteq\R^n$ be a neighborhood of $f(0)$ such that $f(U)\subseteq V$ and $g\in C^k(V, \R)$. Then we have for any $k\in\N$ and $x\in U$ 
	\begin{align*}
	\left(\!\tfrac{d}{dx}\!\right)^k g(f(x)) \!=\! \sum_0 \sum_1 \cdots \sum_k \frac{k!}{\prod_{i=1}^k (i!)^{r_i} \prod_{i=1}^k\prod_{j=1}^nq_{ij}! }  (\partial^\alpha g) (f(x)) \prod_{i=1}^k ( f_1^{(i)}(x))^{q_{i1}} \cdots ( f_n^{(i)}(x))^{q_{in}}
	\end{align*}
	where the sums are over all non-negative integer solutions of the following equations
	\begin{align*}
	& \sum_0 :  r_1 + 2r_2 + \cdots + k r_k = k, \\
	& \sum_1 : q_{11} + q_{12} + \cdots + q_{1n} = r_1 \\
	& \ \ \vdots  \\
	& \sum_k  : q_{k1} + q_{k2} + \cdots + q_{kn} = r_k
	\end{align*}
	and $\alpha := (\alpha_1,\ldots,\alpha_n)$ with $\alpha_j := q_{1j} + q_{2j} + \cdots + q_{kj}$ for $1\leq j \leq n$. 	
\end{lemma}

In the following we seldom need the precise form of Faà di Bruno's formula in the multivariate case but only the fact that there is a universal polynomial $p_k^{(n)}$ with non-negative coefficients independent of $f,g$ such that 
\begin{align}\label{universalpoly}
\left(\tfrac{d}{dx}\right)^k g(f(x)) = p_k^{(n)} (\{\partial_\alpha g\}_{|\alpha|\leq k}, \{ f^{(j)}_i\}_{i=1, \ldots, n, j=1, \ldots, k} )
\end{align}
for all $f \in C ^k (\mathbb R, \mathbb R ^n)$ and $g \in C^k(\mathbb R ^n, \mathbb R).$ Furthermore, $p_k^{(n)}$ is one-homogeneous in the first entries.

\begin{lemma}
	\label{lemma:estimateAnalyticFaaDiBruno}
	Let $k,n \in \N$, $C_x,C_y>0$, $x_j^{(i)}>0$ and 
	\begin{align}\label{eq:lemmaUpperBound}
	y_\alpha \in \left(0,C_y \frac {(\abs{\alpha} +1)!} {r^{\abs{\alpha} +1}}\right)
	\end{align}
	for all $\alpha \in \N_0^n$ with $\abs{\alpha} \le k$, $i \in \{1,\ldots,k\}$ and $j \in \{1,\ldots,n\}$.
	Then, for $C:=\max\{1,C_x\}$
	\[
	p_k^{(n)}(\{y_\alpha\}_{\abs{\alpha}\le k}, \{C_x x_j^{(i)}\}_{\substack{i = 1, \ldots, k\\j =1, \ldots, n}})
	\le C_y p_k^{(n)}\left(\left\{\frac {(\abs{\alpha} +1)!} {\left(\frac r C\right)^{\abs{\alpha} +1}}\right\}_{\abs{\alpha}\le k}, \{x_j^{(i)}\}_{\substack{i = 1, \ldots, k\\j =1, \ldots, n}}\right).
	\]
\end{lemma}
Note that as soon as $y_\alpha = \abs{\partial^\alpha g}$ for some function $g$ which is analytic away from the origin, the upper bound for $y_\alpha$ in \eqref{eq:lemmaUpperBound} automatically holds, see the proof of \cite[Theorem 7.2]{BlattVorderobermeier:2019:CriticalMoebius}.
\begin{proof}
	Using the explicit form, we may compute
	\[
	\begin{split}
	&p_k^{(n)}(\{y_\alpha\}_{\abs{\alpha}\le k}, \{C_x x_i^{(j)}\}_{\substack{i = 1, \ldots, k\\j =1, \ldots, n}})\\
	=&\sum_0 \sum_1 \cdots \sum_k  \frac{k!}{\prod_{i=1}^k (i!)^{r_i} \prod_{i=1}^k\prod_{j=1}^nq_{ij}! }  y_\alpha \prod_{i=1}^k \prod_{j=1}^n (C_x x_j^{(i)})^{q_{ij}}\\
	\le& C_y \sum_0 \sum_1 \cdots \sum_k  \frac{k!}{\prod_{i=1}^k (i!)^{r_i} \prod_{i=1}^k\prod_{j=1}^nq_{ij}!} \frac {(\abs{\alpha} +1)!} {r^{\abs{\alpha} +1}} C_x^{\abs{\alpha}} \prod_{i=1}^k \prod_{j=1}^n (x_j^{(i)})^{q_{ij}},
	\end{split}
	\]
	as the sum of all $q_{ij}$ is exactly $\abs{\alpha}$.
	Estimating $C_x^{\abs{\alpha}} \le C^{\abs{\alpha} +1}$ yields the desired result.
\end{proof}

The explicit formula also allows us to estimate the norm of a Faà di Bruno polynomial in a Banach algebra.
\begin{lemma}
	\label{lemma:estimateFaaDiBrunoBanachAlgebra}
	Let $\XL$ be a Banach algebra with norm $\norm{\cdot}$ in the sense that there is a universal constant $b$ such that for $x,y \in \XL$, $\norm{xy} \le b \norm{x}\norm{y}$.    
	
	Furthermore, let $k \in \N$ and $x_j^{(i)},y_\alpha \in \XL$ for all $\alpha \in \N_0^n$ with $\abs{\alpha} \le k$, $i \in \{1,\ldots,k\}$ and $j \in \{1,\ldots,n\}$.    
	Then,
	\[
	\norm*{p_k^{(n)}(\{y_\alpha\}_{\abs{\alpha}\le k}, \{x_j^{(i)}\}_{\substack{i = 1, \ldots, k\\j =1, \ldots, n}})}
	\le p_k^{(n)}(\{b^{\abs{\alpha}}\norm{y_\alpha}\}_{\abs{\alpha}\le k}, \{\norm{x_j^{(i)}}\}_{\substack{i = 1, \ldots, k\\j =1, \ldots, n}}.
	\]
\end{lemma}
Note that in our application scenario, the product is just pointwise multiplication between functions mapping to $\R$.
\begin{proof}
	The statement follows from the estimate $\norm{\prod_{l=1}^m z_l} \le b^{m-1} \prod_{l=1}^m \norm{z_l}$ for all $z_l \in \XL$ and the fact that the $q_{ij}$ sum up to $\abs{\alpha}$ together with the explicit formula
    of Faa di Bruno's polynomial in Lemma \ref{lem:FaadiBruno}, in particular the one-homogeneity of its first entry.
\end{proof}

\section{An inequality for Sobolev Norms}

\begin{lemma}
    \label{lemma:equivalentSobolevNormsDerivative}
    Let $m>1$ and $f \in H^m(\R/\Z,\R^n)$ with $\int_0^1 f(x) \d x =0$.
    Then,
    \[
        \frac 1 {2\pi} \norm[H^{m-1}]{f'}
        \le \norm[H^m]{f}
        \le \frac 1 {\sqrt{2} \pi} \norm[H^{m-1}]{f'}
        \le \norm[H^{m-1}]{f'}.
    \]
\end{lemma}
\begin{proof}
    By \cite[Proposition~3.1.2~(10)]{Grafakos:2014:ClassicalFourierAnalysis}, we have $\widehat{f'}(k) = 2\pi i k \hat f(k)$ for all $k \in \Z$ and thus for $k \ne 0$,
    \[
        \abs{\hat f(k)}^2 = \frac 1 {(2 \pi k)^2} \abs{\widehat{f'}(k)}^2.
    \]
    Note that
    \[
        \hat f(0) = \int_0^1 f(x)e^{-2\pi \cdot 0 \cdot x} \d x =\int_0^1 f(x) \d x =0
    \]
    with the same holding for $\widehat{f'}(0)$ as $f$ is periodic and continuous.
    Keeping the estimates $k^2 \le 1+k^2 \le 2 k^2$ in mind, we obtain
    \begin{align*}
        \frac 1 {(2\pi)^2} \norm[H^{m-1}]{f'}^2
        &= \sum_{k\in \Z \setminus \{0\}} \frac 1 {(2\pi)^2} (1+k^2)^{m-1} \abs{\widehat{f'}(k)}^2
        = \sum_{k\in \Z \setminus \{0\}} (1+k^2)^{m-1}  k^2 \abs{\hat f(k)}^2\\
        &\le \sum_{k\in \Z \setminus \{0\}} (1+k^2)^m \abs{\hat f(k)}^2
        = \norm[H^m]{f}^2\\
        &= \sum_{k\in \Z \setminus \{0\}} (1+k^2)^m \frac 1 {(2 \pi k)^2} \abs{\widehat{f'}(k)}^2
        \le \frac 2 {(2\pi)^2}  \sum_{k\in \Z \setminus \{0\}} (1+k^2)^{m-1} \abs{\widehat{f'}(k)}^2\\
        &= \frac 1 {2\pi^2} \norm[H^{m-1}]{f'}^2.
    \end{align*}
\end{proof}

    \bibliographystyle{alpha}
    \bibliography{Analyticity-Menger}
\end{document}